\documentclass[reqno, a4paper]{amsart}
\usepackage{amssymb, amscd, amsfonts,  amsthm}
\usepackage[mathscr]{eucal}
\usepackage{graphicx}
 
\newtheorem{theorem}{Theorem}[section]
\newtheorem{lemma}[theorem]{Lemma}
\newtheorem{corollary}[theorem]{Corollary}
\newtheorem{proposition}[theorem]{Proposition}

\theoremstyle{definition}
\newtheorem{definition}[theorem]{Definition}

\newtheorem{example}[theorem]{Example}
\newtheorem{problem}[theorem]{Problem}
 
 \newcommand{\restrict}{\,{\mathbin{\vert\mkern-0.3mu\grave{}}}\,}

\newcommand{\luk}{\L u\-ka\-s\-ie\-wicz}
\newcommand{\commento}[1]{}
\newcommand{\remove}[1]{}

\DeclareMathOperator{\McNn}{\mathcal M([0,1]^{\it n})}
\DeclareMathOperator{\McNtwo}{{\mathcal M}([0,1]^{\rm 2})}

\DeclareMathOperator{\McNk}{\mathcal M([0,1]^{\kappa})}
\DeclareMathOperator{\McN}{\mathcal M}
\DeclareMathOperator{\conv}{\rm conv}
\DeclareMathOperator{\den}{\rm den}

\DeclareMathOperator{\I}{[0,1]}
\DeclareMathOperator{\cube}{[0,1]^{\it n}}
\DeclareMathOperator{\aff}{\rm aff}

\DeclareMathOperator{\cl}{\rm cl}
\DeclareMathOperator{\interior}{\rm int}
\DeclareMathOperator{\relint}{\rm relint}

\DeclareMathOperator{\id}{\rm id}
\DeclareMathOperator{\ret}{\mathcal{M}}
\DeclareMathOperator{\zret}{\mathcal Z}
\DeclareMathOperator{\range}{\rm range}

\DeclareMathOperator{\gen}{\rm gen}

\DeclareMathOperator{\ind}{\iota}

\newcommand{\McNr}{\McN_{\mathbb R}}

 \title[Retractions of free MV-algebras and unital $\ell$-groups]
{Retractions of free MV-algebras and unital $\ell$-groups}

\author{Leonardo Manuel Cabrer and Daniele Mundici }
\thanks{ The first author was supported by the FWF Austrian Science Fund project START Y544-N23 (Project Leader: Prof. Agata Ciabattoni)} 

\address[L.M. Cabrer]{Department of Statistics,
Institute of Computer Languages\\ 
Technische Universit\"at Wien\\
Favoritenstrasse 9-11, 
A-1040\\ Wien \\
Austria}
\email{ leonardo.cabrer@tuwien.ac.at }

\address[D. Mundici]{Department of
Mathematics and Computer Science  ``Ulisse Dini'' \\
University of Florence\\
Viale Morgagni 67/a \\
I-50134,  Florence \\
Italy}
\email{ mundici@math.unifi.it }

\date{\today}

\begin{document}


\begin{abstract} \,\,\,\,\,\,  
%Several 

A number of
papers 
deal with the problem of counting the
number  
of retractions of a structure  $S$  onto a substructure  $T.$ 
In the particular case when  $S$ is a free algebra, 
this number is $\geq 1$  iff $T$ is projective.  
In this paper we consider the case 
when $T$ is a projective
 lattice-ordered abelian group with a distinguished
strong order unit,  or equivalently, a projective  MV-algebra. 
 Let  \,$A$\, be a retract of the free
$n$-generator 
MV-algebra  $\McNn$ of McNaughton functions on $\I^n$. 
We prove that the 
 number $\mathsf{r}(A)$ of retractions
 of  $\McNn$ onto $A$ is finite 
 if, and only if,   the maximal spectral space  $\mu_A$ is homeomorphic
to a (Kuratowski) closed domain $M$ of $\I^n$,  in the sense that  $M=\cl(\interior(M))$. 
 Further, the closed domain condition  is decidable and  
 $\mathsf{r}(A)$ is
 computable, once a retraction onto $A$ is 
 explicitly given. Thus every  
finitely generated projective MV-algebra  $B$  comes equipped  
with a new invariant  
$\iota(B)=\sup\{\mathsf{r}(A) \mid \mbox{$A\cong B$ 
\mbox{\,\,for\,} $A$ a retract of $\McN(\I^{k})$} \},$
 where $k$ is the smallest number of generators of $B$.
We compute  $\iota(B)$ for many projective MV-algebras $B$ 
considered  in  the literature.  Various  problems 
concerning retractions of free MV-algebras
 are shown to be decidable. 
Via the $\Gamma$ functor,
our results and computations automatically transfer to finitely generated
projective abelian $\ell$-groups with a distinguished strong unit.  
\end{abstract}

\maketitle

%%%%%%%%%%%%%%%%%%%%%%%%
%t
%
%
%
\section{Foreword}
Several papers 
deal with the problem of counting the
number $\mathsf{r}(T)$  
of {\it retractions} (= idempotent endomorphisms) 
of a structure  $S$  onto a substructure  $T \subseteq S.$
See,  e.g., 
\cite{one, usp, sch, war},  \cite[p.174]{eva}, 
\cite[p.122]{mon}.
In the particular case when  $S$ is a free algebra, 
$\mathsf{r}(T)\geq 1$ iff $T$ is projective.  

In this paper we will compute   $\mathsf{r}(T)$
when $T$ is a   
projective  MV-algebra or equivalently, a projective
{\it unital $\ell$-group}, which is short for
``lattice-ordered abelian  group with a distinguished
strong order unit''.  
As a particular case
of the equivalence $\Gamma$ established in
 \cite[Theorem 3.9]{mun-jfa}, finitely 
 presented   MV-algebras are categorically equivalent to 
finitely presented 
unital $\ell$-groups.
Further, both categories are dually equivalent to
{\it  rational polyhedra},
i.e., finite unions of simplexes with rational vertices in the same
euclidean space  $\mathbb R^n,\,\,\,n=1,2,\dots$, 
with morphisms given by  {\it $\mathbb Z$-maps},
  i.e, piecewise-linear maps $f$
with a finite number of linear pieces,  such that each linear
piece of $f$  has integer coefficients, \cite{marspa},
\cite[\S 3]{mun11}. 
The synergy between these three categories
has received increasing attention in the last few years,
\cite{cab}--\cite{carrus}, \cite{dubpov},
 \cite{fuc}, \cite{marspa}, \cite{mun11}.
 
%cabmun-au, cabmun, cabmun-jlc,

%Over the last few years  there has been a growing interest
%in  finitely presented and, especially, in projective    MV-algebras
%and unital $\ell$-groups, which is short for
%``abelian lattice-ordered groups with a distinguished
%strong order unit'',   \cite{cab}--\cite{
%%, cabmun-au, cabmun, 
%%cabmun-jlc,  
%carrus}, \cite{dubpov}, \cite{fuc}, \cite{marspa}, \cite{mun11}.

Differently from  finitely presented $\ell$-groups, 
finitely presented  {\it unital}  $\ell$-groups,
as well as finitely presented MV-algebras  $A$ and their dual
rational polyhedra,
are endowed with  a  wealth of computable  invariants,   
such as:  the number of rational points  of
a given denominator $d=1,2,\dots$ in the maximal spectral space
$\mu_A$,  \cite[Proposition 3.15 and Theorem 4.16]{mun11};\,\, the 
  rational measure  of the $t$-dimensional
part of $\mu_A,\,\,(t=0,\ldots,\dim(\mu_A))$,\,\, \cite[\S 14]{mun11}, \cite{mun-cpc}.
% the number of elements of the smallest base of $A$,
%\cite[\S 6]{mun11}.  NOT COMPUTABLE
None of these
 invariants makes sense for finitely presented $\ell$-groups. 
 
 A new numerical invariant,  the {\it index} $\iota(A)$,  will be
introduced in this paper, by counting the maximum number of
retractions  of a free $n$-generator algebra onto $A$, where
$n$ is the smallest number of generators of $A$.
%\,\,\, the number
%of regular triangulations $\Delta$ of $\mu_A$ such that the denominators of
%the vertices of $\Delta$ do not exceed a given
%upper bound $u$;\,\,\, the cardinality of a
%regular triangulation of $\mu_A$ having the smallest number of simplexes.

%
Not surprisingly, the isomorphism problem for finitely presented
unital   $\ell$-groups is still open, although
Markov's celebrated  unrecognizability theorem \cite{sht} 
is to the effect  that the isomorphism problem  for finitely presented
 $\ell$-groups is undecidable, \cite{glamad}.
 
 Another main point of distinction between $\ell$-groups and
 unital $\ell$-groups is the characterization of finitely generated 
 projectives.  On the one hand, from the  Baker-Beynon duality
 \cite{bak,bey77uno,bey77due}  one easily obtains that   
  finitely generated projective $\ell$-groups  coincide with
   the finitely
presented ones. On the other hand,  
finitely generated projective unital $\ell$-groups (resp., 
finitely generated projective MV-algebras) are a proper
subclass of finitely presented  unital $\ell$-groups (resp., 
finitely presented   MV-algebras).
Their characterization is a tour de force in  
 algebraic topology,  \cite{cab, cabmun}.

In this paper we focus on 
$n$-generator projective
MV-algebras, $n=1,2,\dots,$  using their rich algebraic, geometric,
arithmetic and algorithmic structure. It is well known that
any such MV-algebra   is isomorphic to a retract $A$ of the free
MV-algebra  $\McNn$  of McNaughton functions over the
unit $n$-cube  $\I^n.$
  Let  $\mathsf{r}(A)$ denote the number  of retractions
of $\McNn$ onto $A$.  In
Theorem \ref{theorem:computable-multiplicity}
we prove that    $\mathsf{r}(A)$ is Turing computable.

Following Kuratowski,  \cite[p.20]{eng},  we say that 
a subset $D$ of a topological space $X$ is a
{\it closed domain} in $X$ if  $D$ coincides with the closure of
its interior in $X$, in symbols,  
$\cl(\interior(D))=D$.  
For any finitely generated projective MV-algebra $B$, letting $k_B$
be the smallest number of its generators,  we
 define the {\it index}  $\iota(B)$ as
the sup of all $\mathsf{r}(A)$ as $A$ ranges over retracts of $\McN([0,1]^{k_B})$
isomorphic to $B$. Then in Corollary \ref{corollary:criterion} we prove that 
 $\iota(B)$  is finite iff  the maximal ideal space of $B$
is homeomorphic to a closed domain  in $\mathbb R^{k_B}$.
Depending on $B$, $\iota(B)$ can be an arbitrarily large
finite  number, already
in the two-dimensional case,  (Theorem \ref{theorem:fibonacci}).   
  Various estimates  and computations of the multiplicity and of the index
    are carried on (respectively in \S\S 3-6  and  \S 7), and various related
 problems are shown to be Turing decidable. 
% In particular, 
%the index of every finitely generated free MV-algebra is equal to 1.

Via  the mentioned $\Gamma$ equivalence,   
the results of this paper
 automatically transfer to finitely generated 
projective unital $\ell$-groups.
Anyway,  in this paper we will mostly work in the MV-algebraic
framework, because  all the algebraic
machinery concerning finite presentations and projectives,
 (resp., all the algorithmic  machinery needed to compute invariants) 
 naturally arises from
MV-algebras  (resp., from the underlying  \luk\ calculus
of MV-algebras).
For all necessary background  on MV-algebras 
we refer to the monographs
  \cite{cigdotmun} and \cite{mun11}.

\section{Polyhedra and retracts of  free MV-algebras and unital $\ell$-groups  }
\label{section:introductory}
%%%%%%%%%%%%%%%%%%% 
A {\it rational polyhedron}  $P\subseteq \mathbb R^n$  is the 
union of finitely many simplexes in $\mathbb R^n$ with rational vertices.
By a
$\mathbb Z$-{\it map}  $\zeta\colon P \to \I^m$ we mean a piecewise
linear map where each linear piece has integer coefficients, and the
number of linear pieces is finite.
(Throughout this paper the adjective
``linear'' is understood in the affine sense.)
A $\mathbb Z$-{\it homeomorphism}  $\theta$
of a rational polyhedron $P \subseteq \I^n$ onto a
rational polyhedron $Q \subseteq \I^m$ is a
$\mathbb Z$-map of $P$ onto $Q$ such that
also the inverse $\theta^{-1}$ is a $\mathbb Z$-map. 
A $\mathbb Z$-map $\sigma\colon \I^n\rightarrow  \I^n$ is
said to be  a {\it
$\mathbb Z$-retraction of $\I^n$} if it satisfies the idempotence condition  
$\sigma\circ\sigma = \sigma$. The set
$$
\mathsf{R}_\sigma = \range(\sigma)\subseteq \I^n
$$
  is
said to be  a
$\mathbb Z${\it-retract of} $\I^n$. 
$\mathsf{R}_\sigma$ is
a rational polyhedron, and we have the identity
$\mathsf{R}_\sigma=\{x\in \I^n\mid x=\sigma(x)\}.$ 

For $n=1,2,\dots,$ we let 
 $\McNn$ denote 
the MV-algebra  of   $\I$-valued $\mathbb Z$-maps
defined
over     $\cube$, equipped with the
pointwise operations of the standard MV-algebra  $\I.$
$\McNn$ is a free MV-algebra, which throughout this paper  comes equipped
with the free generating set $\{\pi_1,\ldots,\pi_n\}$, where  
$\pi_i\colon \I^n\to \I$ is   the $i$th coordinate map.
Elements of  $\McNn$ are known as {\it McNaughton functions}.
%The two-element MV-algebra $\{0,1\}$ is free over 0 generators. 

For any MV-term  $q(X_1,\ldots,X_n)$ 
we write    $\hat q\colon \I^n\to \I$ 
for  the   McNaughton function
associated to  $q$.
In the notation of  \cite[\S 3.1]{cigdotmun},
$\hat q$ is written $q^{\McNn}$.
In particular,  $\hat{X_i}$ is the $i$th coordinate function
$\pi_i\colon \I^n\to \I.$
More generally, for any  $n$-tuple  $t=(t_1,\ldots,t_n)$ of MV-terms,
where all  $t_i$ are in the same variables  $X_1,\dots,X_n$,
we let $\hat{t}$ denote the $\mathbb Z$-map 
$(\hat t_1,\ldots,\hat t_n) \colon \I^n\to \I^n.$ 

%As is well known, for any algebra $F$,
%an endomorphism $\epsilon\colon F\to F$ 
%is said to be a {\it retraction} of $F$ if  it is
%idempotent, in the sense that 
%  $\epsilon\circ\epsilon=\epsilon.$
%An algebra $A$ is a  {\it retract} of $F$ if
%$A=\epsilon(F)$ for some retraction $\epsilon$ of $F$.
 %

%
%
%
%Conversely,  every $\mathbb Z$-retraction  DOPPIO
%$\sigma\colon \I^n\to \I^n$ determines the retraction
%$-\circ\sigma\colon \McNn\to \McNn$   given by the stipulation 
%$-\circ\sigma\colon f\in \McNn\mapsto f\circ\sigma\in \McNn.$
%%

Following  \cite{marspa}, let  $\mathcal M$ denote  
  the functor from the category of rational 
polyhedra with $\mathbb Z$-maps to finitely presented MV-algebras,
 \cite[\S 3]{mun11}, \cite{marspa}.  For any rational polyhedron
 $P\subseteq [0,1]^n$, the MV-algebra $\McN(P)$ is defined by
 restricting to $P$ every element of $ \McNn$, in symbols,  
$\McN(P)=\{f\restrict P\mid f\in \McNn\},$ where $\restrict$ denotes
restriction.
Further, the
action of $\mathcal M$ on any $\mathbb Z$-map  $\sigma$ is given by
\begin{equation}
\label{equation:functor}
\McN_\sigma=-\circ\sigma.
%=\overline{\pi_i\mapsto \sigma_i}
% So $-\circ\sigma$ is
%a retraction of $\McNn$. 
\end{equation}

If in particular
$\sigma\colon \I^n\to \I^n$  is a $\mathbb Z$-retraction,
  $\McN_\sigma$ is a retraction that  maps
$\McNn$ onto the
MV-subalgebra 
$\gen(\sigma_1,\ldots,\sigma_n)$  of $\McNn$
generated by $\sigma_1,\ldots,\sigma_n$.
Thus by \eqref{equation:functor},  $\McN_\sigma$ is the
uniquely determined homomorphism of $\McNn$ into $\McNn$
extending  the map  $\pi_i\mapsto \sigma_i,\,\,\,(i=1,\ldots,n).$  
%
%For any $\mathbb Z$-retraction
%$\sigma\colon \I^n\to \I^n$, let 
%\begin{equation}
%\label{equation:asigma}
%A_\sigma=\range(-\circ\sigma)=\gen(\sigma_1,\ldots,\sigma_n).
%\end{equation}
%Let 
% \begin{equation} 
%\label{equation:psigma}
%\mathsf{R}_\sigma=\range(\sigma).
%\end{equation}
%
%

Conversely, for   any   retraction 
 $\epsilon\colon \McNn\to \McNn$,
 the $n$-tuple $\zret_\epsilon=( \epsilon(\pi_1),
\dots,\epsilon(\pi_n))\colon \I^n\to\I^n$
 is a  $\mathbb Z$-retraction of $\I^n.$
 The range   $\mathsf{R}_{\zret_\epsilon}$ of $\zret_\epsilon$
is a rational polyhedron and coincides with the set   
$\{x\in \I^n\mid x=\zret_\epsilon(x)\}.$

It is easy to see that
  two maps  $\ret$ and $\zret$ are inverses of each other, 
\begin{equation}
\label{equation:correspondence}
\ret_{\zret_\epsilon}=\epsilon
\,\,\,\mbox{  and  }\,\,\,\zret_{\ret_\sigma}=\sigma.
\end{equation}
%
%As a matter of fact, 
%$\ret_{\zret_\epsilon}=\overline{\pi_i\mapsto (\zret_\epsilon)_i}
%=\overline{\pi_i \mapsto \epsilon(\pi_i)}=\epsilon.$
%Conversely, 
%$\zret_{\ret_\sigma}=(\ret_\sigma(\pi_1),\dots,\ret_{\sigma}(\pi_n))=$
%$((\overline{\pi_i\mapsto \sigma_i})(\pi_1),\dots, (\overline{\pi_i\mapsto \sigma_i})(\pi_n))
%$$ =
%(\sigma_1,\dots,\sigma_n)=\sigma.$
%%

\noindent
Throughout we let  ${\rm id_X}$ denote
 the identity map on a set $X$.
 By a {\it retract} we mean the range of a retraction.

\begin{theorem}
\label{Theorem:bella}   Let $\sigma=(\sigma_1,\dots,\sigma_n)$ be
a $\mathbb Z$-retraction of  $\I^n$  onto the rational polyhedron
$\mathsf R_\sigma.$ 
Let  $\gen(\sigma_1,\dots,\sigma_n)$ be the retract of 
$\McNn$ associated with~$\sigma$. 

\begin{itemize}
\item[(a)]
The map  $\tau\mapsto \mathsf R_\tau$ 
yields a one-one correspondence
between: 
\begin{itemize}
\item[---] $\mathbb Z$-retractions $\tau=(\tau_1,\dots,\tau_n)$ of
$\I^n$  such that $\gen(\tau_1,\dots,\tau_n)$
$=\gen(\sigma_1,\dots,\sigma_n), $ and

\item[---] rational polyhedra  $Q\subseteq \I^n$ such that
$\sigma\restrict Q\colon Q\cong_{\mathbb Z} \mathsf R_\sigma$.
\end{itemize}

\medskip
\item[(b)]
%as given by \eqref{equation:correspondence}.  
Thus there exists  a
one-one correspondence
between:
\begin{itemize}

\item[---] retractions of $\McNn$ onto the MV-algebra 
$\gen(\sigma_1,\dots,\sigma_n)$.

\item[---] rational polyhedra  $Q\subseteq \I^n$ such that
$\sigma\restrict Q\colon Q\cong_{\mathbb Z} \mathsf R_\sigma$, and

\end{itemize}
\end{itemize}
\end{theorem}

\begin{proof} (a)  Let $\tau\colon \I^n\to\I^n$ be a 
$\mathbb Z$-retraction satisfying the condition 
$$\gen(\tau_1,\dots,\tau_n)=\gen(\sigma_1,\dots,\sigma_n).$$
 Then there are MV-terms $t_1,\ldots,t_n$ and $
 s_1,\ldots,s_n$ such that $\tau_i=t_i(\sigma_1,\ldots,\sigma_n)$ 
 and $\sigma_i=s_i(\tau_1,\ldots,\tau_n)$.
Hence  $\hat{t}=(\hat{t}_1,\ldots,\hat{t}_n)$ and 
$\hat{s}=(\hat{s}_1,\ldots,\hat{s}_n) \colon \I^n\to \I^n$
 are $\mathbb Z$-maps satisfying   
\begin{equation}
\label{equation:alpha-new}
%	\alpha\colon \mathsf{R}_\tau \cong_{\mathbb Z} \mathsf{R}_\sigma,
%	\,\,\,\beta\colon \mathsf{R}_\sigma \cong_{\mathbb Z} \mathsf{R}_\tau,
%	\,\,\,\beta=\alpha^{-1},\,\,\,\,
\sigma=\hat{s}\circ\tau\mbox{ and }\tau=\hat{t}\circ\sigma.
\end{equation}

\medskip
\noindent {\it Claim.}
$\sigma\restrict \mathsf{R}_\tau$ is a
$\mathbb Z$-homeomorphism
onto  $\mathsf{R}_\sigma$
satisfying the identity 
\begin{equation}\label{eq:inverse}
(\sigma\restrict \mathsf{R}_\tau)^{-1}=\tau\restrict \mathsf{R}_\sigma.
\end{equation}

As a matter of fact, let us pick
an arbitrary   $x\in \mathsf{R}_\sigma$.
The identities   $(\sigma\circ\tau)(x) =(\hat{s}\circ\tau\circ\tau)(x)
=(\hat{s}\circ \tau)(x)=\sigma(x)=x$
show that $\sigma\restrict \mathsf{R}_\tau$ is onto 
$\mathsf{R}_\sigma$. Similarly, for all
  $y\in \mathsf{R}_\tau$
  we have
  $(\tau\circ\sigma)(y)=y$.
  It follows  that $\sigma\restrict \mathsf{R}_\tau $ is one-one.
The identity \eqref{eq:inverse} is now immediate, and the
claim is proved.
 
\smallskip
To complete the proof of (a), let us assume that, conversely, 
 $Q\subseteq \I^n$ is a rational polyhedron such that 
 $\sigma\restrict Q\colon Q\cong_{\mathbb Z} \mathsf R_\sigma$.
Let us write $\zeta=(\zeta_1,\dots,\zeta_n)$ 
as an abbreviation of the $\mathbb Z$-homeomorphism 
$(\sigma\restrict Q)^{-1}$ of $\mathsf R_\sigma$ onto $Q$,
$$
\zeta=(\sigma\restrict Q)^{-1}\colon
\mathsf R_\sigma\cong_{\mathbb Z} Q.
$$
Observe  that  $\zeta$ is piecewise linear with integer coefficients,
and is defined over the rational polyhedron  $\mathsf R_\sigma$.
For short,  $\zeta$ is a $\mathbb Z$-map on  $\mathsf R_\sigma\subseteq \I^n$.  
So by  \cite[Proposition 3.2]{mun11}  
we have  a $\mathbb Z$-map   $\overline{\zeta}\colon\I^n\to\I^n$ extending $\zeta.$
By McNaughton theorem, \cite[Theorem 9.1.5]{cigdotmun}, 
 for each  $i+1,\dots,n$,\,\,
  $\overline{\zeta_i}$ is the McNaughton function of some
$n$-variable MV-term.
The composite map $ \rho=\zeta\circ\sigma$
is  a $\mathbb Z$-retraction    of  $\I^n$ onto $Q$, because 
  $\zeta\circ\sigma\circ\zeta\circ\sigma=
\zeta\circ\, {\rm id}_{\mathsf{R}_\sigma}\circ\sigma=\zeta\circ\sigma.$
Since $\mathsf{R}_\sigma$    is a $\mathbb Z$-retract of $\I^n$ then so
is the rational polyhedron $Q=\zeta(\mathsf{R}_\sigma)$.  
From $ \rho=\zeta\circ\sigma$ we get 
$\gen( \rho_1,\dots, \rho_n)\subseteq \gen( \sigma_1,\dots, \sigma_n)$.
From 
$\sigma=\zeta^{-1}\circ \rho=\sigma\circ\rho$ 
we get $\gen( \sigma_1,\dots, \sigma_n)\subseteq \gen( \rho_1,\dots, \rho_n)$. 
Further, by \eqref{equation:alpha-new} and \eqref{eq:inverse} we can write 
$
\zeta\circ \sigma=
\tau\restrict R_{\sigma}\circ\sigma=\tau\circ\sigma=
\hat{t}\circ\sigma\circ\sigma
=\hat{t}\circ\sigma=\tau,
$
and 
$
R_{\zeta\circ \sigma}=\zeta
\circ \sigma(\I_n)
=\zeta(\mathsf{R}_{\sigma})=Q.
$
Thus
 the maps $\tau\mapsto \mathsf{R}_{\tau}$ and 
 $Q\mapsto\zeta\circ \sigma$ are inverse of each other,
 and (a) is proved.

\medskip
(b)  This immediately follows from (a) and  \eqref{equation:correspondence}.  
\end{proof}

For the proof of  
 Theorem~\ref{theorem:finite} below,
  we record the following elementary fact: 
 
\begin{lemma}
\label{Lem:SeparationConComp}
Let $\eta\colon\I^n\to \I^n$ be a $\mathbb Z$-map and $P,Q\subseteq \I^n$ 
be rational polyhedra satisfying the following conditions:

\smallskip
\begin{itemize}
\item[(i)] both $\interior(P)$ and $\interior(Q)$ are connected;

\smallskip
\item[(ii)] $P=\cl (\interior(P))$ and $Q=\cl (\interior(Q))$;

\smallskip
\item[(iii)] $\eta(P)=\eta(Q)$;

\smallskip
\item[(iv)] $\eta\restrict P\colon P\cong_{\mathbb Z}\eta(P)$\,\, 
and \,\,$\eta\restrict Q\colon Q\cong_{\mathbb Z}\eta(Q)$.
\end{itemize} 
Then either $P=Q$ or ${\rm int}(P)\cap{\rm int}(Q)=\emptyset$.
\end{lemma}

 \begin{proof}
By way of contradiction, let us assume
 $P\neq Q$ and there is $x\in \interior(P)\cap\interior(Q).$   
Without loss of generality assume that $y\in P\setminus Q$ for some $y$. 
By (ii),   $P=\cl (\interior(P))$, whence we may insist that  $y\in \interior(P)$.
  Since by (i) the interior of $P$   is an open  connected 
  subset of  $\mathbb R^n$,  it is also path connected.
  (See Figure \ref{figure:boundary}.) 
  Let   $\gamma\colon\I\to P$ be a path
  such that $\gamma(\I)\subseteq \interior(P)$, $\gamma(0)=x$, and $\gamma(1)=y$.  
%Since  $P$ is a polyhedron, we
% may insist that $W$ is a broken line
%with finitely many nodes.  
Since $\gamma$ is continuous and $Q$ is closed, 
 the set $J=\{\delta\in\I\mid \gamma(\delta)\in Q\}\subseteq \I$ is closed.
Let  $\lambda$ be the largest element of $J$. 
From $\gamma(1)=y\notin Q$ we get $\lambda<1$. 
Let $z=\gamma(\lambda)$. Then $z\in \interior(P)$ and  $z\in Q\setminus \interior(Q)$.
\commento{L to D: THis follows only from (iv)}
By (iv), 
%(iii)-(iv), 
$\eta$ maps  $z$ to a point  $\eta(z)$ that simultaneously
belongs to the interior of  $\eta(P)$ 
and  
 to the boundary of  $\eta(Q)$, 
\commento{L to D: Is it worth to be mor precise here?}
which contradicts (iii).
%a contradiction.
\end{proof}

\begin{figure} 
    \begin{center}                                     
    \includegraphics[height=8.0cm]{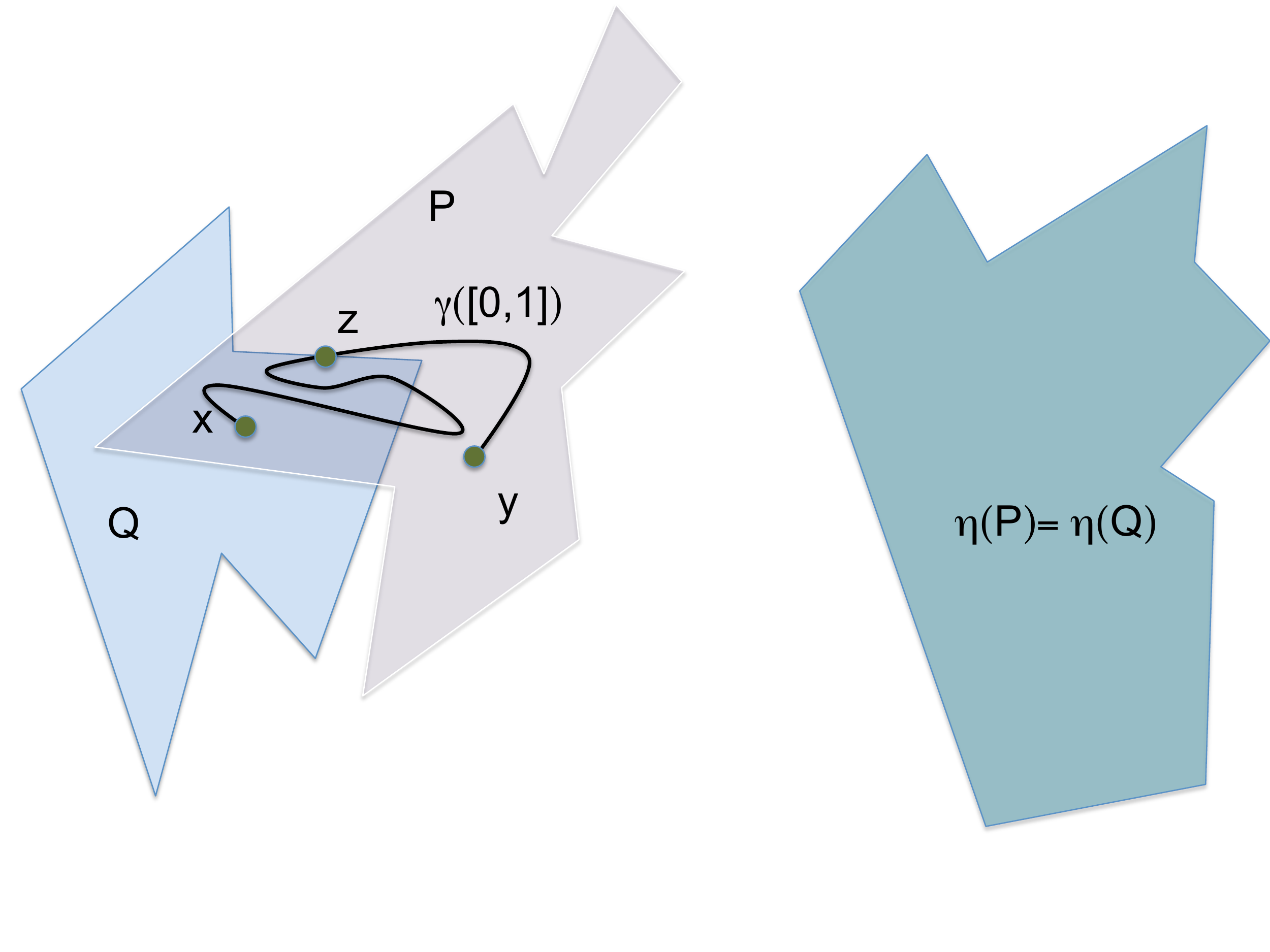}    
    \end{center}                                       
 \caption{\small   The path $\gamma\colon[0,1]\to \interior(P)$
 in the proof of Lemma \ref{Lem:SeparationConComp}
joins 
 $x\in\interior(P)\cap\interior(Q)$ and
 $y\in P\setminus Q$, and  has a nonempty intersection with
 the boundary of  $Q$.}  
    \label{figure:boundary}                                                 
   \end{figure}

 Up to isomorphism, any $n$-generator projective MV-algebra $B$  has the
form $\McN(P)=\{f\restrict P\mid f\in \McNn\}$ for  some  $\mathbb Z$-retract $P$ of  $\I^n$.   
Specifically,  by
 \cite[Theorem 5.1]{cabmun} or   \cite[Proposition 17.5]{mun11},  there is
a $\mathbb Z$-retraction $\sigma$ of $\I^n$
 such that $B\cong\McN(\mathsf R_\sigma)\cong
 \range(\mathcal M_\sigma)=\range(-\circ\sigma)$.
By \cite[Corollary 4.18]{mun11},  
the $\mathbb Z$-retract $Q=\mathsf R_\sigma$ is homeomorphic to the
maximal spectral space  $\mu(B).$  If another
$\mathbb Z$-retract $Q'$ of $\I^n$ is chosen such that
$B\cong\McN(Q')$, then $Q$ is $\mathbb Z$-homeomorphic to $Q'$
(\cite[Corollary 3.10]{mun11}). Thus in particular  $Q$ is a 
closed domain 
in $\I^n$ iff so is  $Q'.$  (See \cite[p.20]{eng} for this terminology,
going back to Kuratowski.) This state of affairs can be
unambiguously described by saying that the maximal spectral
space $\mu_B$ is a closed domain in $\I^n$.

\begin{theorem}
\label{theorem:finite}
Suppose  $A$ is a
 retract  of $\McNn$  and  $\mu_A$ is a closed domain
in $\I^n.$ 
Then the number of retractions of $\McNn$ onto $A$
is finite.  
\end{theorem}

\begin{proof} 
Let us choose  a 
retraction  $\epsilon$  of $\McNn$ onto $A$, along with its associated  
$\mathbb Z$-retraction  $\mathcal Z_\epsilon=\sigma$
%=(\sigma_1,\ldots,\sigma_n)$ 
as given by 
\eqref{equation:correspondence}.
Since $\mathsf{R}_\sigma$ is a polyhedron (it is compact and) 
the connected components of $\interior(\mathsf{R}_\sigma)\subseteq \I^n$ are finitely many. Let  
$O_{\sigma,1},\ldots,O_{\sigma,k}\subseteq \interior(\mathsf{R}_\sigma)\subseteq \I^n$ 
be the list of these connected components. 

With reference to the notation
\eqref{equation:functor} for the functor $\mathcal M$, 
let $\zeta$ be a $\mathbb Z$-retraction of  $\I^n$  such that
$\McN_\zeta$ is a retraction of $\McNn$ onto $A$.   
By Theorem~\ref{Theorem:bella},
 $\sigma\restrict\mathsf{R}_\zeta$  is $\mathbb Z$-homeomorphism onto $\mathsf{R}_\sigma$. 
Therefore,  $\interior(\mathsf{R}_\zeta)$ has $k$ connected components
$O_{\zeta,1},\ldots,O_{\zeta,k}$, and we can write 
$\sigma(O_{\tau,j})=O_{\sigma,j}.$
We have  $\mathbb Z$-homeomorphisms
\begin{equation}\label{eq:restrictZhomeo}
\sigma\restrict\cl(O_{\zeta,j}) \colon \cl(O_{\zeta,j}) \to \cl(O_{\sigma,j}),\,\,\,(j=1,\dots,k). 
\end{equation}

Let the family
$\mathcal O$ of open sets in $\I^n$ be defined by
$$\mathcal O =\{O_{\zeta,j} \mid j=1,\dots,k,
 \mbox{ and the map $ -\circ\zeta$ is  a retraction
of $\McNn$ onto $A$}\}.$$
Let $O_{\zeta,j},O_{\zeta',j'}\in\mathcal O$. If $O_{\zeta,j}\neq O_{\zeta',j'}$,
 then either $j\neq j'$ or $\zeta\neq \zeta'$. If $j\neq j'$,
  then $\sigma (O_{\zeta,j})\cap\sigma(O_{\zeta',j'})=\emptyset$, 
whence $O_{\zeta,j}\cap O_{\zeta',j'}=\emptyset$. If $j=j'$, then  $\zeta\neq \zeta'$. 
From Lemma~\ref{Lem:SeparationConComp}, 
(with $\eta=\sigma$, 
$P=\cl(O_{\zeta,j})$, 
and $Q=\cl(O_{\zeta',j'})$) it follows that
 $O_{\zeta,j}\cap O_{\zeta',j'}=\emptyset$. Therefore, the elements of $\mathcal O$ are
 pairwise  disjoint.

Since $\mathbb Z$-homeomorphisms preserve
the Lebesgue measure of $n$-dimensional  polyhedra in $\I^n$ 
(\cite[Lemma 14.3]{mun11}, \cite[Theorem 2.1(iii)]{mun-dcds}), by \eqref{eq:restrictZhomeo} each $O_{\zeta,j}\in\mathcal O$   
 has the same ($n$-dimensional)  Lebesgue measure  
as $O_{\sigma,j}$,
because $O_{\sigma,j}$ has the same Lebesgue measure as $\cl(O_{\zeta,j})$. 
Let $O_{\sigma,j}$ be chosen among $O_{\sigma,1},\ldots,O_{\sigma,k}$
as having the smallest $n$-dimensional
Lebesgue measure. Say that $\lambda$ is its measure. 
Since  the elements of $\mathcal O$ are pairwise disjoint, we have
\begin{equation}
\label{equation:uniform-bound}
\mbox{number of elements in }\mathcal O\leq \lfloor {1}/{\lambda}\rfloor
=\max\{l\in\mathbb Z\mid l\leq {1}/{\lambda}\}.
\end{equation}
By Theorem~\ref{Theorem:bella},  the number $\mathsf{r}(A)$ of 
retractions of  $\I^n$ onto~$A$ satisfies the inequality
\begin{equation}
\label{equation:upperbound}
\mathsf{r}(A) \leq \binom{\lfloor 1/\lambda\rfloor}{k}. 
\end{equation}
This completes the proof. 
\end{proof}

\medskip

Throughout we let   $\McNr(\I^n)$ denote the unital $\ell$-group of
piecewise linear  functions $f\colon \I^n\to \mathbb R$,  where each linear
piece of $f$ has integer coefficients.  In view of
 the categorical equivalence  $\Gamma$ between unital $\ell$-groups and
MV-algebras, \cite[Theorem 3.9]{mun-jfa},
such notions as ``free unital $\ell$-group'' and ``finitely presented
$\ell$-group'' make perfect sense, not only as the $\Gamma$-correspondents
of free and finitely presented MV-algebras, but also from the categorical viewpoint, 
(respectively see \cite[Corollary 4.16]{mun-jfa} and
\cite[Remark 5.10]{carrus}.) 

 The maximal spectral space
$\mu_G$
of every unital $\ell$-group $(G,u)$
 is canonically homeomorphic to the maximal spectral space
of  its associated MV-algebra  $\Gamma(G,u),$
\,\,\cite[\S 7.2]{cigdotmun}.  Precisely as in the case of
MV-algebras,  it makes perfect 
mathematical sense to say  that  $\mu_G$  is a closed
domain in $\I^n$.

By   \cite[Theorem 4.15]{mun-jfa}, 
$\Gamma(\McNr(\I^n))=\McNn$. Thus
by \cite[\S 7.2]{cigdotmun}, up to unital $\ell$-isomorphism 
every finitely generated projective unital $\ell$-group  has the
form $\McNr(P)=\{f\restrict P\mid f\in \McNr(\I^n)\}$ for some  $n=1,2,\dots$ and
    $\mathbb Z$-retract $P$ of 
$\I^n$.

\smallskip

%We refer to \cite[Corollary 4.16]{mun-jfa}  for the freeness properties of 
% the unital $\ell$-group  $\McNr(\I^n)$.

\begin{corollary}
\label{theorem:finite-bis}
Given a retract $(G,u)$ of $\McNr(\I^n)$,  suppose $\mu_G$ is a closed domain
in $\I^n.$ 
Then the number of retractions of $\McNr(\I^n)$ onto $(G,u)$
is finite.  
\end{corollary}

\begin{proof} Immediate from Theorem \ref{theorem:finite},
 using the
preservation properties of the $\Gamma$ equivalence,  
\cite[\S 7.2]{cigdotmun}.
\end{proof}

\medskip
%%%%%%%%%%%%%%%%%%%%%%%%%%%%
\section{The index of a  projective MV-algebra and of a unital $\ell$-group}
%%%%%%%%%%%%%%%%%%%%%%%%%%%%

%A {\it representation into $\McN(\I^l)$}  of a finitely
%generated projective MV-algebra $B$ is an isomorphism
%$\phi$  of $B$ onto a retract $A$ of $\McN(\I^l)$.
%By assumption, $B$ has a representation  into $\McNn$, with  $n$
%the smallest possible number of generators of $B$.
%Then trivially, for  no $m<n$ 
%there is a representation of $B$ into $\McNm$. 

%We say that $\phi$ is a {\it core} representation  if for no $m<n$ 
%there is a representation of $B$ into $\McNm$.

%\begin{proposition} 
%\label{proposition:above}
%Suppose  $B$ projective MV-algebra. 
%Then $B$ has a representation  in $\McNn$, with  $n$
%the smallest possible number of generators of $B$.
%For no $m<n$ 
%there is a representation of $B$ into $\McNm$.
%\end{proposition}
%\begin{proof} 
%Let $m$ be the smallest number of generators of $B$ and $n$ be such that $B$ admits a core representation $A$ in $\McNm$.
%
%Let $\epsilon$ be a retraction from $\McNn$ onto $A$. Then $\epsilon(\pi_1),\ldots,\epsilon(\pi_n)$ is a set of generators of $A$. Hence $m\leq n$.
%
%Conversely, let $b_1,\ldots,b_m\in B$ be a set of generators then assignment $\pi_i\mapsto b_i$ extends to a homomorphism $h$ from $\McNm$ onto $B$. Since $B$ is projective there exists $f\colon B\to  \McNm$ such that ${\rm id}_{B}=h\circ b$. Hence $A=f(B)$ is a representation of $B$. Hence $m\leq n$.
%\end{proof}

\begin{definition} 
The {\it multiplicity}  
$\mathsf{r}(A)$ of a retract $A$ of $\McNn$ is
the  number of distinct retractions of $\McNn$ onto 
$A$ if this number is finite,  and $\infty$ otherwise. 
The {\it index} \/ $\ind(B)\in\{1,2,\ldots\}\cup\{\infty\}$
 of a finitely generated projective MV-algebra
$B$ is the supremum   of the multiplicities of all retracts
$A \cong B$ of  $\McNk$, with $k$
 the smallest number of generators of $B$.
 One similarly defines 
 the index of
finitely generated  projective unital $\ell$-groups.
\end{definition}

\begin{proposition}
\label{proposition:raffinanda}
(a) Let  $P\subseteq \I^n$   be a   $\mathbb Z$-retract and 
a  closed domain in $\I^n$.   Let 
 $m$ be the maximum number of $\mathbb Z$-homeomorphic 
 pairwise disjoint copies of $P$ in $\I^n$.  Then   
$
\iota(\McN(P))\geq m.
$

(b) An upper bound for the index $\iota(\McN(P))$ 
 is given by \eqref{equation:upperbound}.
%
%\medskip
%\noindent  (b)  Let  $P\subseteq \I^n$   be a   $\mathbb Z$-retract and 
%a  closed domain in $\I^n$. 
% Then   we have the upper bound
%$$
%  \iota(\McN(P))\leq 1/ \lambda(P),
%$$
%where $\lambda(P) $  is the $n$-dimensional Lebesgue measure of $P$.
%Further,   $\iota(\McN(P))=\iota(\McNr(P))$. 
%
%\bigskip
%\noindent  
%(c) More generally, let  $B$ be an $n$-generator projective MV-algebra
%whose  maximal spectral space $\mu_B$   is
%homeomorphic to a
%closed domain  $D$  in $\I^n$.
%Then  $\iota(B)\leq 1/ \lambda(D)$.
\end{proposition}

\begin{proof}
(a) Our assumption $P=\cl(\interior(P))$ ensures that
  $n$ is the smallest number of generators of $\McN(P)$.
 As a matter of fact, if  $\McN(P)$ had  $n-1$ generators
 (absurdum hypothesis) then by \cite[Theorem 3.6.7]{cigdotmun} 
 $\McN(P)$  would be isomorphic to an MV-algebra of the
 form $\McN(X)$ for some closed subset $X$ of $\I^{n-1}.$
 By \cite[Corollary 4.18]{mun11} the maximal spectral space
 $\mu_{\McN(X)}$ is homeomorphic to $X,$ whence  its dimension
 is $\leq n-1$.
 On the other hand,  from the isomorphism $\McN(P)\cong \McN(X)$
 we get the homeomorphism  $P\cong X$, so $\dim(P)\leq n-1,$
 thus contradicting the assumption that $P$ is a closed domain in
 $\I^n.$

 Let $Q_1, \,\,Q_2,\ldots,Q_m$ be a (maximal)  set of
 pairwise disjoint  $\mathbb Z$-homeomorphic 
copies of $P$ in $\I^n$. 
Since by \cite[Corollary 3.10]{mun11}  $\McN(Q_1)\cong \McN(P)$
and the index is an isomorphism invariant, we may assume
$Q_1=P$  without loss of generality. If $m=1$ we have nothing
to prove. So assume  $m\geq 2.$ For each  $i=2,\ldots,m$
there is a  $\mathbb Z$-homeomorphism $\eta_i$
of $Q_i$ onto $Q_1.$ For completeness let us set  $\eta_1=\id_P.$
Since the $Q_j$ are pairwise disjoint  ($j=1,\dots,m$) the set
$
\bigcup_{j=1}^m \eta_j
$
is a $\mathbb Z$-map of $\bigcup_{j=1}^m Q_j$ onto $P$.
By \cite[Proposition 3.2(ii)]{mun11} there is a $\mathbb Z$-map 
$
\eta\colon [0,1]^n\to[0,1]^n
$
simultaneously extending each $\eta_j.$
Pick a $\mathbb Z$-retraction  $\sigma$ 
 of $\I^n$ onto $P$. Then
the composite map  $\sigma\circ \eta$ is a 
$\mathbb Z$-retraction of $\I^n$ onto $P$, and for
each  $j=1,\dots,m$ the restriction 
\commento{L to D: Small error corrected and some extra details added.}
$\sigma\circ\eta\restrict Q_j =\sigma\circ \eta_j = \eta_j$
%$\rho\circ\eta\restrict Q_j$
%
%
is a $\mathbb Z$-homeomorphism of 
$Q_j$ onto $P$.  By Theorem \ref{Theorem:bella}, the multiplicity of
the retract  $A=\gen(\sigma_1,\ldots,\sigma_n)$ 
is $\geq m.$   By \cite[Lemma 3.6]{mun11} 
$A\cong \McN(P)$, whence
the desired conclusion follows by definition of the index,
recalling that $n$ is the smallest number of generators of
$\McN(P)$.

(b) By \cite[Lemma 14.3]{mun11} or \cite[Theorem 2.1(iii)]{mun-dcds},
 $\mathbb Z$-homeomorphisms preserve
the Lebesgue measure of $n$-dimensional  polyhedra in $\I^n$.
By \cite[Corollary 3.10]{mun11}, 
$\McN(P)\cong\McN(Q)$ implies $P\cong_\mathbb Z Q.$
\end{proof}

\begin{corollary}
\label{corollary:index-free}
 The index 
of every finitely generated  free MV-algebra  is
 $1$.    Similarly, for  every $n=1,2,\dots$, the index of  the free unital
 $\ell$-group $\McNr(\I^n)$ is  $1$.
In particular, the two-element MV-algebra
$\{0,1\}$ is the free MV-algebra over $0$ generators.
Its index is equal to $1$.   
\end{corollary}

The following example explains why
in the definition of the index of $B$ we restrict to those
isomorphic copies of $B$ which are retracts of $\McNk$  
{\it with
$k$ the smallest number of generators of $B$.}
As in \cite[p.11]{mun11}, or \cite{mun-cpc}, for any
rational point  $r\in \mathbb R^n$, the {\it denominator}
$\den(r)$ is defined by
\begin{equation}
\den(r) =
\mbox{ least common denominator of the
coordinates of $r$.}
\label{equation:den} 
\end{equation}

\begin{example} 
For $n \geq 1$ let ${\rm cyl}(n,\McN(\I))\subseteq \McNn$
be the isomorphic copy of $\McN(\I)$ obtained  by
cylindrifying each  $f\in\McN(\I)$
 into the function  $c\in \McNn$ given
by   $c(x,x_2,\ldots,x_n)=f(x)\,\,\mbox{ for all } (x,x_2,\ldots,x_n)\in \I^n.$ 
By Corollary \ref{corollary:index-free} 
 the index of the free
MV-algebra $\McN(\I)$ is 1.  We {\it claim} 
that   the
multiplicity of its isomorphic copy ${\rm cyl}(2,\McN(\I))$  is  $\infty$.
Let the  $\mathbb Z$-retraction  $\xi=(\xi_1,\xi_2)\colon
\I^2\to \I^2$ be given by $\xi_1(x,y)=x,\,\,\,\xi_2(x,y)=0$.
$\xi$ projects any point of the unit square onto the $x$-axis.
A direct inspection shows that
  $\xi$ preserves the denominator of a
 rational point of $(x,y)\in \I^2$ iff 
 the denominator of $y$ is a divisor of 
 the denominator of $x$. This is the case,
 in particular, when the point $(x,y)$ belongs to
 the graph $W$ of a McNaughton function $f$ in $\McN(\I)$,
 \,\,because every linear piece of $f$ has integer
 coefficients. 
By \cite[Proposition 3.15]{mun11},
  $\xi$ acts $\mathbb Z$-homeomorphically
over the  broken line  $W\subseteq \I^2$.
  There are countably many such broken
lines  $W$, one for each $f\in \McN(\I).$ By Theorem~\ref{Theorem:bella}(c)
 there are countably
many retractions of $\McN(\I^2)$ onto ${\rm cyl}(2,\McN(\I)).$
Thus the multiplicity of  ${\rm cyl}(2,\McN(\I))$  is  $\infty$,
and our claim is proved.
One similarly proves that
 the multiplicity of ${\rm cyl}(n,\McN(\I))$  is  $\infty$
for each  $n\geq 2$. 
%
%To evaluate the index of  ${\rm cyl}(n,\McN(\I))$  we must
%refer to its {\it core} representations, or equivalently,  to the
% core representations of  $\McN(\I)\cong {\rm cyl}(n,\McN(\I)).$
As already noted in Corollary \ref{corollary:index-free}, 
 $\iota(\McN(\I))=1$
whence
 $\iota({\rm cyl}(n,\McN(\I)))=1$ for each $n$. 
\end{example}

 \begin{figure} 
 \begin{center}                                     
    \includegraphics[height=8.0cm]{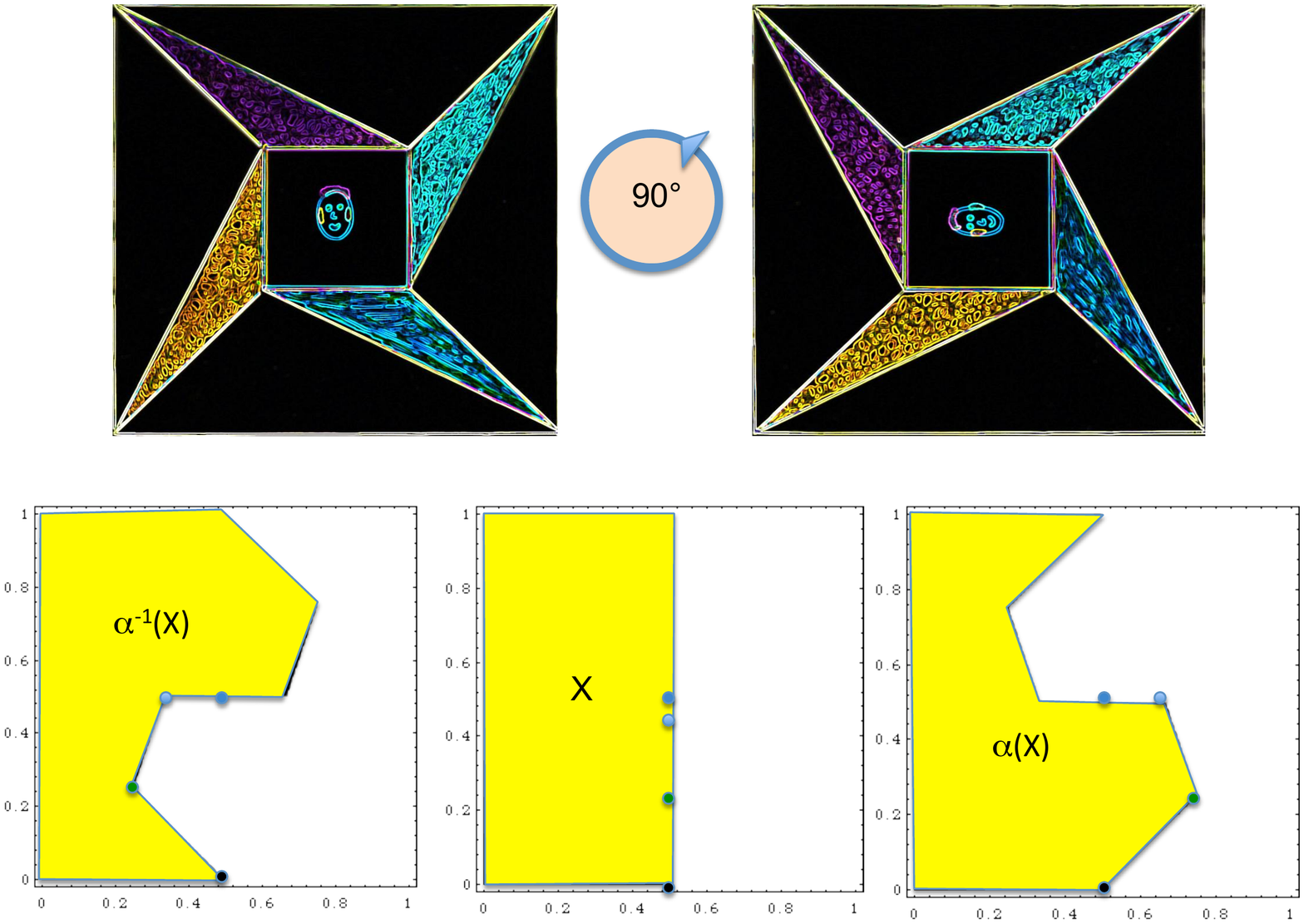}    
    \end{center}                                       
 \caption{\footnotesize Panti's $\mathbb Z$-homeomorphism
  $\alpha\colon \I^2\to \I^2$ and the image  $\alpha(X)$
 of the rectangle $X=\{(x,y)\in\I^2\mid x\leq 1/2\}$.  As explained in
 \cite{dingripan}, $\alpha$ rotates counterclockwise by
  $90^\circ$  the inner square
  with vertices $(1/3,1/3),(1/3,2/3),(2/3,1/3),(2/3,2/3)$;
  the rubber edges of the eight triangles in the picture
are modified accordingly.  
The perimeter of the square  
%$\I^2$
is kept fixed by $\alpha$, and so is the central point
$(1/2,1/2)$. The $\mathbb Z$-retraction $\sigma\colon(x,y)\mapsto x\wedge \neg x$ 
sends  $\I^2$ onto $X$.   The map
$\alpha\circ\sigma$ sends $\I^2$  onto $\alpha(X)$, but is not a 
$\mathbb Z$-retraction.
The $\mathbb Z$-retraction
$\tau=\alpha\circ\sigma\circ \alpha^{-1}$  sends $\I^2$ onto $\alpha(X)$. 
So both $X$ and  $\alpha(X)$  are 
$\mathbb Z$-retracts of $\I^2$. 
%Note that $\alpha(X)$ and $\alpha^{-1}(X)$ are the mirror images of each other
%around the horizontal line  $Y=1/2$. 
%A more general example
%is discussed in Proposition  \ref{example:tapioco}
}  
    \label{figure:tapioco}                                                 
   \end{figure}

The proof of the following result is immediate:
 
\begin{proposition}
\label{example:tapioco}
Let $\sigma=(\sigma_1,\ldots,\sigma_n)$ be a $\mathbb Z$-retraction of $\I^n$,
and  $\alpha$   a $\mathbb Z$-homeomorphism
of $\I^n$ onto $\I^n.$ 
Then the range of the composite map  $\alpha\circ\sigma$ 
is a $\mathbb Z$-retract, and so is its $\mathbb Z$-homeomorphic copy
$\mathsf{R}_\sigma\subseteq \I^n$. 
%An example of a  $\mathbb Z$-retraction  $\tau$ of $\I^n$ 
%such that  $Q=\mathsf{R}_\tau$
%is given by $\alpha\circ\sigma\circ\alpha^{-1}.$
If  $\,\mathsf{R}_\sigma$ is
$n$-dimensional,
% $\gen(\sigma_1,\ldots,\sigma_n)$ is a retract of
% $\McNn$, and 
$n$ is the smallest
number of generators of the
retract $\gen(\sigma_1,\ldots,\sigma_n)$ of $\McNn$. 
Letting $\tau=\alpha\circ\sigma\circ\alpha^{-1}$,
it follows that  $\tau$ is a
$\mathbb Z$-retraction of $\I^n$, and the 
two isomorphic  retracts  $\gen(\sigma_1,\ldots,\sigma_n)$ and 
$\,\gen(\tau_1,\ldots,\tau_n)$ have equal multiplicities and equal indexes.
\end{proposition}

 {Figure \ref{figure:tapioco}}  shows the special case
of Proposition \ref{example:tapioco} for    $n=2$, where
$\alpha$ is
Panti's $\mathbb Z$-homeomorphism,  \cite{dingripan} and
$\sigma=\pi_1\wedge\neg\pi_1\colon \I^2\to\I^2$.

\medskip
The effective 
computability  of the index of a one-generator projective MV-algebra is
taken care of by the following easy result:

\begin{proposition}
\label{proposition:exceptional-case}
Let $B\not\cong\{0,1\}$ be a one-generator projective MV-algebra.
\begin{itemize}
\item[(a)]
For a unique rational $0< r \in\I$ we have the
isomorphism $B\cong \McN([0,r]).$ Then
$\iota(B)\in\{1,2\}$. 
Further, 
$\iota(B)=2$  iff
$r\leq 1/2.$

\item[(b)] 
In equivalent
algebraic-topological  terms,   
$\iota(B)=2$, unless the maximal spectral space
$\mu_{B}$   contains 
an element $\mathfrak m$ such that  $B/\mathfrak m
\cong\{0,1/2,1\}$ and
$\mu_{B}\setminus\{\mathfrak m\}$  is disconnected---in which case
$\iota(B)=1$.
\end{itemize} 
\end{proposition}

\begin{proof}  (a) The first statement is a particular case
of \cite[Proposition 17.5]{mun11}, upon noting that 
every $\mathbb Z$-retract of $\I$ is $\mathbb Z$-homeomorphic
to $\McN([0,r])$ for some $r\in\mathbb Q\cap \I.$
Further, 
$r>0,$ for otherwise $B$  would be isomorphic to the 
two-element MV-algebra.
In case  $r>1/2$
the measure-theoretic argument in the proof of   Theorem \ref{theorem:finite} 
shows that  $\iota(A)=1$. 
On the other hand, if  $r\leq 1/2$, 
the only other rational polyhedron in $\I$ which is
$\mathbb Z$-homeomorphic to $[0,r]$ is
$[1-r,1]$. By Theorem
\ref{Theorem:bella}(b) and Proposition
\ref{proposition:raffinanda},
$\iota(\McN([0,r]))=2.$

% only retract 
%$D\not=\McN([0,r])$ of
% $\McN(\I)$ that is   isomorphic
%to   $\McN([0,r])$ is $D=\McN([1-r,1])$.
%
%
%
% $x\wedge \neg x$
%is  a $\mathbb Z$-retraction $\sigma$ of $\I$ onto $[0,r]$ that is also
%a $\mathbb Z$-homeomorphism of $[1-r,1]$ onto $[0,r]$.
%Thus the retract  of $\McN(\I)$ given by $\{f\in \McN(\I)\mid f=g\circ \sigma \}$ 
%has multiplicity equal to 2. 

(b) This is just a reformulation of part (a) in the light of the
spectral theory of MV-algebras, \cite[\S 4.5]{mun11},  and the
duality between finitely presented MV-algebras and
rational polyhedra, \cite[\S 3]{mun11}, \cite{marspa}.
\end{proof}

While the index is invariant under isomorphisms,
in the following example we present two isomorphic retracts of
 $\McN(\I)$ having different multiplicities.

\begin{example}
\label{example:noninvariant-multiplicity}  
Let the $\mathbb Z$-retraction of $\I$ onto $[0,1/2]$ be given by  
$\sigma(x)=x\wedge \neg x.$ The retract 
$A=\gen(\sigma)=\gen(\pi_1\wedge \neg\pi_1)\subseteq
\McN(\I)$  is the MV-algebra of all one-variable McNaughton
 functions  $f$  such
 that $f(1-x)=f(x).$ 
Since the restriction of  $\sigma$ to $[1/2,1]$
is a $\mathbb Z$-homeomorphism onto  $[0,1/2]$ and $(\sigma\restrict[1/2,1])^{-1}
=\pi_1\vee \neg \pi_1$, by  Theorem ~\ref{Theorem:bella}
the map $\McN_{\pi_1\vee \neg \pi_1}$ 
 is a second  retraction
$\McN(\I)$ onto $A$.
Moreover,
 $\McN_\sigma$ and $\McN_\rho$ are the only two
 retractions of  $\McN(\I)$ onto $A.$ 
 Thus $\mathsf{r}(A)=2$. 
Let $\tau\colon \I\to \I$ be given by
$\tau(x)=(x\wedge \neg x)\wedge ((\neg x\oplus\neg x)\odot(\neg x\oplus \neg x)).$
 Then $\tau$ is a $\mathbb Z$-retraction of $\I$ onto
 $[0,1/2]$.
Let  $B=\gen(\tau)$.
We have  $A\cong B \cong \McN([0,1/2])$. 
 For no other
segment  $J$ other than $[0,1/2]$ it is the case that
$\tau\restrict J$ is a $\mathbb Z$-homeomorphism of $J$ onto
 $[0,1/2]$.  By Theorem~\ref{Theorem:bella},  $\mathsf{r}(B)=1$. 
 \end{example}

\medskip
%%%%%%%%%%%%
\section{When the maximal spectral space is not  a closed domain in $\I^n$}
%%%%%%%%%%%%%

%We let  $\Sigma^{\max}(d)$ denote
%the set of those $d$-simplexes
%$S$ of $\Sigma$ which are not properly contained
%in any simplex of $\Sigma$. Any such $S$
%is said to be {\it maximal} in $\Sigma$.

For any  rational  $m$-simplex
$T = \conv(v_0, \ldots, v_m) \subseteq \mathbb R^{n}$,
let us display
each vertex
$v_j$ of $T$  as  $(a_{j1}/b_{j1}, \ldots, a_{jn}/b_{jn})$,
for uniquely determined
\commento{L to D: are we ever using this notation? Wouldn't it be simpler to write 
``For any rational vertex $v=(v_1,\ldots,v_n)\in \mathbb R^{n}$ we let the  {\it homogeneous correspondent}
$\tilde{v}$  of
${v}$ be defined by 
$\tilde{v} = \den(v)(v_1,\ldots,v_n,1)\in \mathbb{Z}^{n+1}$''. And so avoid all the complicated subindexes needed to distiguish the vertices of $T$.}
%
%
%\commento{L to D: For this to hold we need $\conv(v_0, \ldots, v_m) \subseteq [0,1]^n$ not $\conv(v_0, \ldots, v_m) \subseteq \mathbb R^n$. But this would be different from the rest of the conditions in this introductory words. We should decide which condition we prefer. I think the simplest thing to do is to reverse  $0\leq a_{jt} \leq b_{jt}$ to just  $b_{jt} > 0$\\
%D to L: Of course, thank you; see how I fixed it, following your suggestions  }
%
%
integers $ a_{jt}, b_{jt}\,\,\,
(t=1,\ldots,n)$ such that
$b_{jt} > 0$.
%$\,\,\gcd(a_{jt},b_{jt})=1$.
%The least common multiple of
%$\{b_{j1},\ldots,b_{jt}\}$
%     is called the \emph{denominator}
%of $v_j$, and is denoted
%$\den(v_j)$.
%The  {\it denominator $\den(S)$ of  $S$}
%is defined by
%$\den(S)=\den(v_0)\cdots\den(v_m).$
With the notation of  \eqref{equation:den}
we let the  {\it homogeneous correspondent}
$\tilde{v}_j$  of
${v}_j$ be defined by 
$$\tilde{v}_j = \den(v_j) (a_{j1}/b_{j1},
\ldots, a_{jn}/b_{jn}, \, 1) \in {\mathbb Z}^{n+1}. $$
%
%%
 % This vector
 % is {\it primitive}, i.e., minimal
 % (as an integer nonzero vector) along its {\it ray}
 % $\{\mu\tilde{v}_j\in \mathbb R^{n+1}\mid \mu\geq 0\}$.
 %%
Conversely,
${v}_j$ is said to be the {\it affine correspondent}
of  $\tilde{v}_j$.

An   $m$-simplex $U = \conv(w_{0}, \ldots,w_{m})
\subseteq \mathbb R^{n}$
is said to be
\emph{regular}  if it is rational and
the set of integer vectors
$\{\tilde{w}_0, \ldots, \tilde{w}_m\}$
can be extended to
a basis of the free abelian group
${\mathbb Z}^{n+1}$.

For every   simplicial complex
$\Sigma$ the  point-set union of the simplexes
of $\Sigma$ is called the {\it support}
of $\Sigma$, and is denoted
$|\Sigma|$.
We also say that  $\Sigma$ is a 
\emph{triangulation} of $|\Sigma|$.
A   simplicial complex
is said to be  a \emph{regular
triangulation} (of its support) if
all its simplexes are regular.
Regular
   triangulations
    (called ``unimodular'' in \cite{mun-dcds})
     are the affine counterparts of
the regular, or nonsingular,  fans of toric algebraic
geometry, \cite{ewa}.
Suppose $\Sigma$ and $\Theta$
are two simplicial complexes with the same
support, and every simplex of $\Theta$ is
 contained in some simplex of $\Sigma$.
 Then $\Theta$ is said to be
 a  {\it subdivision} of  $\Sigma$.

 Let $\Sigma$ be a simplicial
 complex  and $b \in |\Sigma|
 \subseteq \mathbb R^{n}$.
Following \cite[III, 2.1]{ewa}, the simplicial complex
 $\Sigma_{(b)}$ is obtained
  by the following  procedure:
 replace every simplex $S \in \Sigma$ containing $b$ by the set
of all  simplexes of the form $\conv(b,F)$, where
 $F$ is any face of $S$ that does not contain $b$.
 The subdivision  $\Sigma_{(b)}$ of $\Sigma$ is known as the 
  {\it blow-up $\Sigma_{(b)}$ of $\Sigma$ at
 $b$}.

 %
%$\Sigma_{(b)}$ is a simplicial complex.
 %The inverse of a blow-up is called a {\it blow-down}.

%%

For any  $m\geq 1$ and regular $m$-simplex
$U =\conv(w_{0},\ldots,w_{m})\subseteq \mathbb R^{n}$ the
{\it Farey mediant of $U$}   is the affine
correspondent
of the  vector $\tilde{w}_0+\cdots+\tilde{w}_m
\in \mathbb Z^{n+1}$.
%,  where each
%$\tilde{w}_i$ is the homogeneous correspondent
%of ${w}_i$.
In the particular case when $\Sigma$
is a regular triangulation and
$b$ is the {\it Farey mediant} of
a  simplex  of $\Sigma$,   the blow-up
$\Sigma_{(b)}$ is  regular.

%%%end of differentiable

\medskip
 
The short proof of the following proposition is a template
for the main construction in the proof of Theorem
 \ref{theorem:converse},
yielding a converse of 
Theorem \ref{theorem:finite}.

\begin{proposition}
\label{proposition:counterexample}
%Let $\sigma$ be a $\mathbb Z$-retraction of $\I^n\,\,\,(n\geq 2)$. 
%Let $n=2,3,\dots$.
%Dropping in Theorem \ref{theorem:finite} the assumption that
%  $\mu_A$  is a closed domain in $\I^n$,
%the conclusion of the theorem fails: 
There is a retract  of $\McNtwo$
having  an infinite index.
\end{proposition}

\begin{proof} 
Let $L$ be the union of the two edges  $\conv((0,1),(0,0))$ and 
$\conv((1,0),(0,0))$. Let  $\rho=(\rho_1,\rho_2)\colon [0,1]^2\to L$
be the $\mathbb Z$-retraction of $[0,1]^2$ onto $L$ given by
\[
\rho(x,y)=(x\ominus y, y\ominus x)=\begin{cases}
(0,y-x)&\mbox{if } y\geq x;\\
(x-y,0)&\mbox{if } x\geq y.
\end{cases}
\]
A direct inspection shows that $\rho$ sends each point  $(x,y)\in [0,1]^2$
to the point of $L$ whose coordinates are  $x-\min(x,y)$ and $y-\min(x,y).$
Geometrically, $\rho$ moves down by 45 degrees in the south-west direction
each point $(x,y)\in [0,1]^2$, by  subtracting the same quantity  $\min(x,y)$ to each coordinate.

\medskip
\noindent{\it Claim.} The MV-algebra $A=\gen(\rho_1,\rho_2)\subseteq \McN(\I^2)$ has infinite multiplicity. 

As a matter of fact, 
(see Figure \ref{figure:infinity}) 
for each integer  $p\geq 3$ let the broken line $W_p\subseteq
\I^2$ be  the union of the segment
$\conv((0,1),(0,0))$ with the three segments
$$
  \conv\left((0,0),(\frac{2}{p},\frac{1}{p})\right),\,\,\,
\conv\left((\frac{2}{p},\frac{1}{p}), (\frac{1}{p-1},0)\right),\,\,\,
\conv\left((\frac{1}{p-1},0),
(1,0)\right).
$$
It is not hard to check that
 $\rho$ is a $\mathbb Z$-homeomorphism of $W_p$ onto $L$. 
To this purpose one 
notes that the triangle  $\conv((0,0),(\frac{2}{p},\frac{1}{p}),(\frac{1}{p-1},0))$
is the union of the regular triangles 
$\conv((0,0),(\frac{2}{p},\frac{1}{p}),(\frac{1}{p},0))$
$\conv((\frac{1}{p-1},0),(\frac{2}{p},\frac{1}{p}),(\frac{1}{p},0))$.  
Further
\begin{itemize}
\item $\rho$ fixes
the segment  $\conv((0,1),(0,0)$;
\item $\rho$ maps
 $\conv((0,0),(\frac{2}{p},\frac{1}{p}))$ one-one onto
 $\conv((0,0),(\frac{1}{p},0))$;
 \item 
 $\rho$ maps
 $\conv((\frac{2}{p},\frac{1}{p}), (\frac{1}{p-1},0))$ 
one one onto  $\conv((\frac{1}{p},0),(\frac{1}{p-1},0))$;
\item $\rho$ fixes $\conv((\frac{1}{p-1},0),(1,0))$.
\end{itemize}
By \cite[Lemma 3.7, Proposition 3.15]{mun11},
 $\rho$ is an invertible
$\mathbb Z$-map of  $W_p$ onto $L$.
To see that the multiplicity of $A$ is infinite,  recall Theorem
\ref{Theorem:bella}, and let $p$ range over all integers $\geq 3$.
Having thus settled our claim, the proof of the proposition is
complete. 
\end{proof}

%\begin{corollary}
%%Let $\sigma$ be a $\mathbb Z$-retraction of $\I^n\,\,\,(n\geq 2)$. 
%Let $n=2,3,\dots$.
%Dropping in Corollary  \ref{theorem:finite-bis} the assumption that
%  $\mu_G$  is a closed domain in $\I^n$,
%the conclusion of the corollary  fails for some retract $(G,u)$ of $\McNr(\I^n)$.
%\end{corollary}

\begin{figure} 
\begin{center}                                     
    \includegraphics[height=8.0cm]{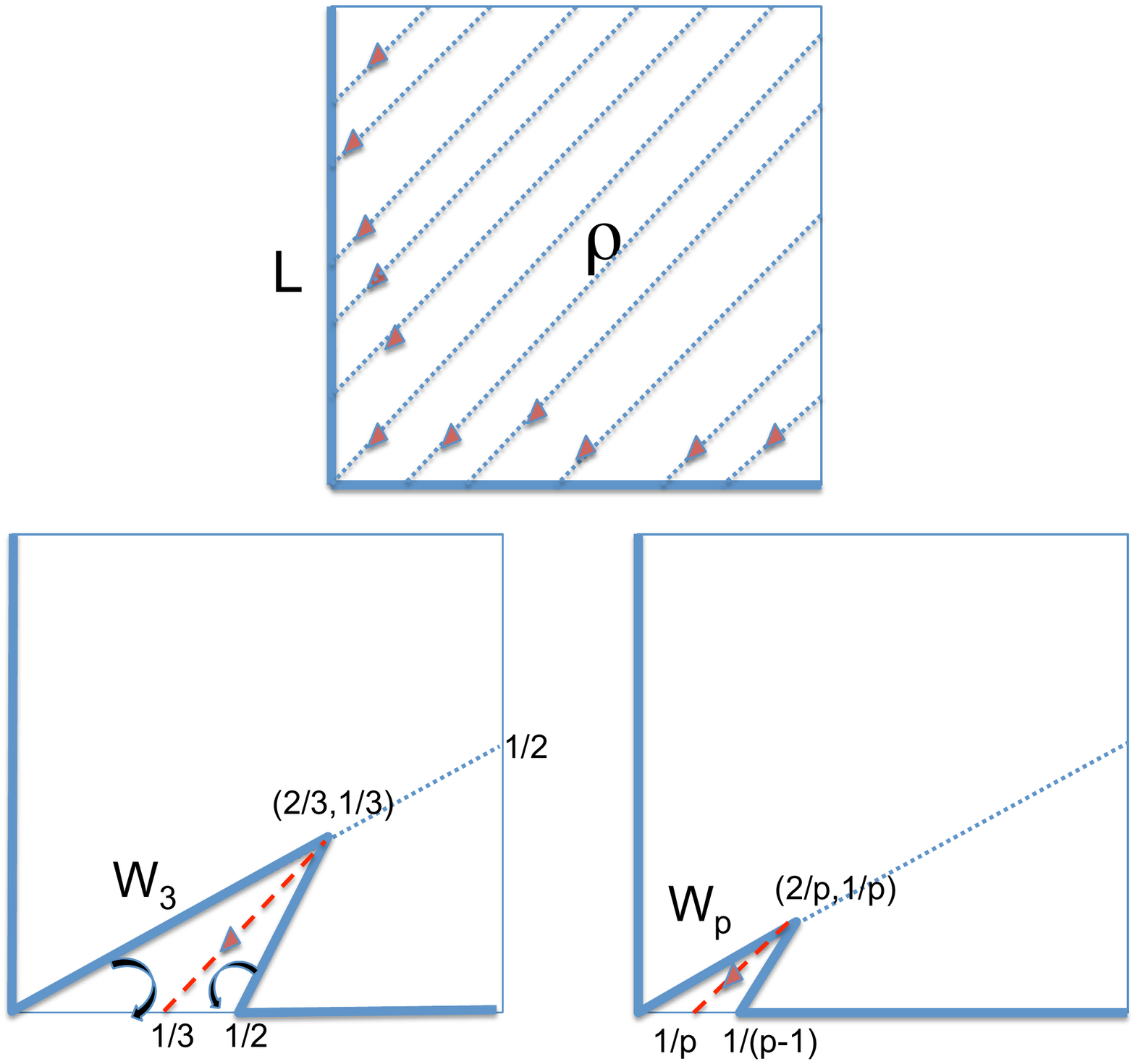}    
    \end{center}                                       
 \caption{\small   
   The $\mathbb Z$-retraction  $\rho$
 in the proof of
Proposition  \ref{proposition:counterexample}.
 Each broken line
 $W_p,\,\,\,p\geq 3,$  is mapped
 $\mathbb Z$-homeomorphically onto $L$ by  $\rho$.   
 The index of the MV-algebra  $\McN(L)$ is $\infty$.}  
    \label{figure:infinity}                                                 
   \end{figure}

\medskip

 Following \cite[Definition 4.1]{cabmun}, a triangulation
 $\Delta$ of a rational polyhedron $P$ is said to be
 {\it strongly regular} if it is regular and for each
 maximal simplex $T$ of $\Delta$ the greatest
 common divisor of the denominators of the vertices
 of $T$ is equal to 1. 
  $P$  is called {\it strongly regular} if it has
 a strongly regular triangulation. Then
every regular triangulation of $P$ is
strongly regular (\cite[Remark 5.1]{cabmun}).
Every $\mathbb Z$-retract of $\cube$ is strongly regular,
\cite[Theorem 5.2(iii)]{cabmun}.

 \begin{theorem}
 \label{theorem:converse}
Fix   a  $\mathbb Z$-retraction  $\rho=
(\rho_1,\ldots,\rho_n)$ of  $\I^n$.  Let $P=\mathsf R_\rho$
be the range of $\rho$.
 If  some (equivalently, every) triangulation of  $P$ contains
 a maximal $m$-simplex with $m<n$ then  the number
 of retractions of $\McNn$ onto  $\gen(\rho_1,\ldots,\rho_n)$ 
 is   infinite. 
  \end{theorem}
  
  \begin{proof}  Since $P$
is a $\mathbb Z$-retract of
  $\I^n$) then  $\McN(P)$ projective,   \cite[Theorem 5.1]{cabmun},
  \cite[Proposition 17.5(ii)]{mun11}.
By Theorem   \ref{Theorem:bella}
  there is a one-one correspondence between the set
of  retractions of $\McNn$ onto the MV-algebra  
  $\gen(\rho_1,\ldots,\rho_n)\subseteq \McNn$  and the set
  of  rational polyhedra  $R\subseteq \I^n$ such that the restriction
$\rho\restrict R$ is a $\mathbb Z$-homeomorphism
of  $R$ onto  $P$. Let us say that any such  $R$ is a
{\it $\mathbb Z$-homeomorphism domain} for $\rho.$
So it suffices to show that the number of such domains $R$ is
infinite.

Let $\Delta$ be a regular triangulation of $P$. The existence of
$\Delta$ follows from
  \cite[Corollary 2.10]{mun11}. By assumption,
$\Delta$ has a maximal $m$-simplex $T$ with  $m<n.$
It follows that  $m\geq 1,$ for otherwise by \cite[Theorem 5.2(i)-(ii)]{cabmun}
the $\mathbb Z$-retract  $P$ would coincide with a vertex of $\I^n$, so $\McN(P)$ is the free MV-algebra over 0 generators,
and $n=m=0$, which is impossible.

Since $P$ is a $\mathbb Z$-retract
of $\I^n$, $P$ is   strongly regular,
 \cite[Theorem 5.2(iii)]{cabmun}.
 Thus,   for  some prime number $p$ there exists
 a rational point  $c$ of denominator $p$ with   
\begin{equation}
\label{equation:odot}
c\in \relint T.
\end{equation}  
Actually, such $c$ exists for all sufficiently large primes $p$, because
$T$ is a strongly regular  $m$-simplex with $m> 0$.

As another consequence of the strong regularity of $P$,
 the affine hull $\aff(T)\subseteq \mathbb R^n$ of $T$
 contains some integer point of $\mathbb Z\subseteq \mathbb R^n,$ \cite[Theorem 4.17]{cab-forum}.
Then the construction of \cite[Lemma 5]{cabmun-etds}
 yields integer points  $j_0,\ldots,j_m \in \mathbb Z^n$
such that  $\aff(T) = \aff(\conv(j_0,\ldots,j_m))$ and the $m$-simplex
$I=\conv(j_0,\ldots,j_m) \subseteq \mathbb R^n$ is regular and contains
$c$ in its relative interior.  

Let   $\mathcal G_n={\rm GL}(n,\mathbb Z)\ltimes \mathbb Z^n$
denote the 
$n$-dimensional affine group over the integers. By 
 \cite[Lemma 1]{cabmun-etds}
some function $\gamma\in \mathcal G_n$ maps 
 $\aff(T)$ one-one onto the $m$-dimensional space
$$
F_m=\{x=(x_1,\ldots,x_n) \in \mathbb R^n\mid x_{m+1}=\cdots=x_n=0\}.
$$ 
Thus the $m$-simplex  $\gamma(I)$ lies in $F_m$, and  we can write without
loss of generality
$$
\gamma(I)=\conv(0,(1,\underbrace{0,\ldots,0}_{n-1 \,\,{\rm zeros}}),
(0,1,\underbrace{0,\ldots,0}_{n-2 \,\,{\rm zeros}}),\ldots,
(0,\ldots,0,1,\underbrace{0,\ldots,0}_{n-m \,\,{\rm zeros}})).
$$
Since $\gamma $ is a (linear)  $\mathbb Z$-homeomorphism,
  it preserves denominators of rational points
  and maps
regular simplexes one-one onto regular simplexes,
 \cite[Proposition 3.15]{mun11}. 
%Thus in particular, the $m$-simplex  $\gamma(I)$ is regular. 
Let us display the
point $c'=\gamma(c)$ as follows:
$$
c'=(c'_1,\ldots,c'_m,\underbrace{0,\ldots,0}_{n-m\,\,zeros})=(a_1/p,\ldots,
a_m/p,0,\ldots,0)
$$ 
for suitable relatively prime integers  $0\leq a_1,\ldots,a_m\leq p.$ Note that
$\den(c')=\den(c)=p.$  By \eqref{equation:odot},
\begin{equation}
\label{equation:relintbis}
c'\in \relint(\gamma(T)).
\end{equation}
We next define the point  $l\in \mathbb R^n$ by

\begin{equation}
\label{equation:ell}
l=(c'_1,\ldots,c'_m,1/p,\underbrace{0,\ldots,0}_{n-m-1\,\,zeros})
=(a_1/p,\ldots,
a_m/p,1/p,0,\ldots,0).
\end{equation}
Permuting, if necessary, the coordinates in $\mathbb R^n$, for all sufficiently
large primes $p$ we can safely assume 
\begin{equation}
\label{equation:lorrain}
\gamma^{-1}(l) \in \I^n.
\end{equation}
Since $P$ is a polyhedron and $T\in \Delta$, then by \eqref{equation:odot}
for all sufficiently small $\epsilon > 0$ the closed ball  $B_{\epsilon, c}$ of radius
$\epsilon$ centered at $c$ satisfies the condition
\begin{equation}
\label{equation:isolation}
B_{\epsilon,c} \cap P \subseteq T.
\end{equation}
 The affine transformation  $\gamma$ sends $B_{\epsilon,c}$ one-one onto
 an $n$-dimensional ellipsoid  $\gamma(B_{\epsilon,c})$  containing
 the point  $\gamma(c)$ in its relative interior. Further, by 
 \eqref{equation:isolation} we can write
  \begin{equation}
\label{equation:stella}
\gamma(B_{\epsilon,c}) \cap \gamma(P) \subseteq \gamma(T).
\end{equation}
The map  $\rho'= \gamma \circ \rho\circ \gamma^{-1}$ is a $\mathbb Z$-retraction
of the $n$-parallelepiped 
$\gamma(\I^n)$ onto the rational polyhedron  $\gamma(P).$  By \eqref{equation:stella},
all points sufficiently close to $c' $ are mapped by  $\rho'$ into points lying in 
the $m$-simplex $\gamma(T).$  For all sufficiently small $\epsilon>0$
the piecewise linear  map $\rho'$ is linear over $\gamma(B_{\epsilon, c})$.
A continuity argument recalling \eqref{equation:relintbis} ensures that
 the point  $l^*=\rho'(l)$ lies in the relative interior of $\gamma(T)$, because
 $l$ tends to $c'$ as $p$ tends to $\infty$.

The De Concini-Procesi theorem in the version of 
\cite[Theorem 5.3]{mun11}  (or the affine version of the desingularization
procedure of \cite[p.70]{ewa})
   yields a regular
triangulation  $\nabla$ of  $\gamma(T)$ such that  $l^*$  is a vertex of some
simplex of $\nabla.$  The set $S$ of $m$-simplexes of $\nabla$ is now defined by
$$
S=\{B\mid B \in \nabla \mbox{\,\,is an $m$-simplex having $l^*$ among its vertices}\}.
$$
Fix now  $B\in S$ and write
$
B=\conv(v_0,v_1,\ldots,v_m)$ for suitable points 
$v_i\in \mathbb R^n.$
%We  refer to  \cite[\S 2.1]{mun11} for
%the basic properties of the
%  homogeneous correspondent of a rational point.
For each $i=0,\dots,m$ let us display the homogeneous 
correspondent
$\tilde v_i\in \mathbb Z^{n+1}$ 
 of vertex  $v_i$ as follows:
$$
\tilde v_i= (b_{i1},\dots,b_{im}\underbrace{0,\dots,0}_{n-m\,\,{\rm zeros}},d_i).
$$
From the regularity of $B\in S\subseteq \nabla$
we get 
\begin{equation}
\label{equation:first-regular}
\det
\left(
\begin{tabular}{ccccc} 
$b_{01}$ & $b_{02}$ & $\dots$ & $b_{0m}$ &$d_{0}$ \\
$b_{11}$ & $b_{12}$ & $\dots$ & $b_{1m}$ &$d_{1}$ \\
$\dots$ & $\dots$ & $\dots$ & $\dots$ &$\dots$ \\   
$b_{m1}$ & $b_{m2}$ & $\dots$ & $b_{mm}$ &$d_{m}$ 
  \end{tabular}
  \right)  
  = \pm 1.
\end{equation}
Recalling \eqref{equation:ell} we can similarly write
$$
\tilde l 
=(a_1,\ldots,
a_m,1,\underbrace{0\dots,0}_{n-m-1{\rm \,\,zeros}}, \,p).
$$
Let $\mathbf{b}_1,\ldots,\mathbf{b}_m, \mathbf{d}$
be the column vectors of the integer matrix  
 \eqref{equation:first-regular}. 
We then have  
\[
\det
\left(
\begin{tabular}{cccccc} 
$\mathbf{b}_1$ & $\mathbf{b}_2$ & $\dots$ & 
$\mathbf{b}_m$ &0&$\mathbf{d}$ \\  
$a_1$ & $a_2$ & $\dots$ &  $a_m$ & $1$ & $p$ 
  \end{tabular}
  \right)  
  = \pm 1,
\]
showing the regularity of the 
$(m+1)$-simplex   $P_B= \conv(B,l)$
for every $B\in S.$
 The basis of the pyramid  $P_B$ is the $m$-simplex
 $B\subseteq F_m.$  The lateral $m$-surface
of  $P_B$ is the point set union of all $m$-simplexes of  $P_B$ having $l$ as a vertex.
Let the $(m+1)$-dimensional pyramid  $P_S\subseteq \mathbb R^n$  be defined by
$$
P_S= \bigcup_{B\in S} P_B.
$$
Its basis  $B_S$  is the point set union  of the bases $B$'s, for all $B\in S.$ 
Let  $L^*_B\subseteq L_B$  be  obtained by stripping $L_B$ of
all $m$-simplexes of $P_B$ having $l^*$ as a vertex. Then the lateral
surface $L_S$ of $P_S$  is given by
$$
L_S = \cl\bigcup_{B\in S} L^*_B = \bigcup_{B\in S}\cl(L^*_B).
$$
Since the $\mathbb Z$-retraction  $\rho'$ is linear over the ellipsoid
$\gamma(B_{\epsilon,c})$ then $\rho'$ maps $L_S$ one-one
onto $B_S.$ Intuitively,
$\rho'$  collapses the lateral surface of $P_S$  one-one onto its basis
$B_S.$  This can be directly verified for each  $B\in S$, noting
that  $\rho'$ maps $\cl(L^*_B)$  one-one onto $B$.  
%Injectivity follows by the linearity 
%properties of $\rho'$ and
%equality of the dimension of $L_S$ and $B_S$.  
Since  $\rho'$  preserves the
denominators of the vertices of each $m$-simplex  $P_B$, and
  $P_B$ is  regular,  then by \cite[Proposition 3.15]{mun11}
$\rho'$ maps 
$\mathbb Z$-homeomorphically   $\cl(L^*_B)$   
  onto  $B$.  Thus $\rho'$ maps 
$\mathbb Z$-homeomorphically  $L_S$   
 onto $B_S.$
Further,  $\rho'$ is identity over $\gamma(P)\setminus B_S$, whence
$\rho'$ is a $\mathbb Z$-homeomorphism of   $(\gamma(P)\setminus B_S)  \cup L_S$
onto $\gamma(P).$

In conclusion,  
%$\rho'$ is a  $\mathbb Z$-retraction of $\gamma(\I^n)$
%onto  $\gamma(P)$, and 
the set $(\gamma(P)\setminus B_S)  \cup L_S$ is
a $\mathbb Z$-homeomorphism domain of $\rho'.$  Going back via $\gamma^{-1}$
we see that  the  $\mathbb Z$-retraction   $\rho$ sends
$(P\setminus \gamma^{-1}(B_S))\cup \gamma^{-1}(L_S)$
$\mathbb Z$-homeomorphically onto $P$.  (Note that
  \eqref{equation:lorrain}
ensures that $\gamma^{-1}(L_S)$ lies 
in the $n$-cube).
The choice of $c\in \relint(T)$ and of the large prime $p$ being arbitrary, 
it follows  that
there are infinitely many  $\mathbb Z$-homeomorphism domains of $\rho.$
By Theorem  \ref{Theorem:bella}  the number of retractions of 
$\McNn$ onto  $\gen(\rho_1,\ldots,\rho_n)$ is infinite. 
 \end{proof}

 Combining the foregoing theorem with 
 Theorem \ref{theorem:finite}  we immediately obtain:
 
 \begin{corollary}
 \label{corollary:criterion} 
 Let $k$ be the smallest number of
 generators of a  finitely generated projective
 MV-algebra  $B$. Then the index of $B$ is finite iff
 the maximal spectral space of $B$ is  homeomorphic to
 a regular domain in $\I^k$. 
 \end{corollary}

 \begin{proof}
 Identify $B$ with $\McN(P)$ for some
 $\mathbb Z$-retract $P$ of $\I^k.$
 
 \smallskip
 $(\Rightarrow)$ 
 If  the maximal spectral space of $B$ is not  homeomorphic to
 a regular domain in $\I^k$, then the same holds for its
  homeomorphic copy
  $P.$  As a consequence,
every (equivalently, some) triangulation $\Delta$ of $P$ contains some
maximal $l$-simplex with $l<k.$  By the foregoing theorem, 
the index of $B$ is infinite. 

\smallskip
 $(\Leftarrow)$ If  the maximal spectral space of $B$ is  homeomorphic to
 a regular domain in $\I^k$, then so is its homeomorphic copy
  $P.$  Now apply Theorem \ref{theorem:finite}. 
 \end{proof}

\bigskip
 \section{Arbitrarily high finite index}

\begin{theorem}
\label{theorem:fibonacci}
 For every  $j=1,2,\dots$ there is
 retract   $A_j$  of  $\McNtwo$ such that 
 the  maximal spectral space of $A_j$  
is a closed domain in $\I^2$ 
and $j<\iota(A_j)\in\mathbb Z$.
\end{theorem}

\begin{proof} For each rational point  $x=(x_1,x_2)\in \I^2$ let
$d=\den(x)$
\commento{L to D: Do we need this repeated definition of $\den$? }
 be the least common multiple of the denominators of  $x_1$ and $x_2$.
Then for uniquely determined integers  $n_1,n_2$ we can write  $x_1=n_1/d$
and $x_2=n_2/d$. Throughout this proof we will specify  $x$ in terms of its
{\it homogeneous integer coordinates} as in \cite[\S 2.1]{mun11}.
Identifying $x$ with its homogeneous correspondent we will write   
\begin{equation}
\label{equation:homogeneous}
x=(n_1/d,n_2/d) = [n_1,n_2,d].
\end{equation}
We will also write  $o$ for the origin $[0,0,1]$ in $\mathbb R^2$. 

The proof
amounts to a construction of  
$\mathbb Z$-retracts  $\sigma^{(1)}, \sigma^{(2)} ,\dots$
of $\I^2$ such that the multiplicity 
of the retract  $A_n= \range(-\circ \sigma^{(n)})$
is $>2^n.$
We assume familiarity with regular triangles, 
regular triangulations, and Farey blow-ups  \cite[\S 2.2, \S
5.1]{mun11}.

\bigskip
\noindent{\it  Step $0$.}  

Let the regular triangles  
$U_1,V_1\subseteq \I^2$ be defined by  
$$V_1=\conv(o, [1,1,1],[0,1,1])
\mbox{ and }
U_1=\conv(o,[1,0,1],[1,1,1]).$$
  Let  $\zeta^{(1)}\colon V_1\to
U_1$ be the unique linear extension of the map 
$$o\mapsto o,\,\,[1,1,1]\mapsto [1,1,1],  \,\,
[0,1,1]\mapsto [1,0,1].$$
 By 
\cite[Lemma 3.7, Corollary 3.10]{mun11},\,\,  $\zeta^{(1)}$ is a
$\mathbb Z$-homeomorphism  of $V_1$ onto $U_1$, in symbols,
$\zeta^{(1)}\colon V_1\cong_{\mathbb Z} U_1.$ 
Next let $\rho^{(1)}=\sigma^{(1)}=\zeta^{(1)}\cup \id_{U_1}$. Then  $\sigma^{(1)}$
is a $\mathbb Z$-retraction of $\I^2$ onto $U_1,$  acting $\mathbb Z$-homeomorphically
over $V_1$.  By Theorem~\ref{Theorem:bella},
 the multiplicity of the retract  $A_1=\range(-\circ \sigma^{(1)})$ 
is equal to $2^1.$

\medskip 
The {\it Fibonacci sequence}  $1,1,2,3,5,8,13,\dots$ be defined by
\begin{equation}
\label{equation:fibonacci}
\mathsf{F}_1=1,\,\,\mathsf{F}_2=1,\,\,\mathsf{F}_{n+1}=\mathsf{F}_{n}+\mathsf{F}_{n-1}.
\end{equation}

\bigskip
\noindent{\it  Step $1$.}

Let $\Sigma_1$ be the regular simplicial complex
given by $U_1$ and all its faces.
Let $b_1$ be the Farey mediant of the edge
of $U_1$ opposite to the origin $o$.
Then the blow-up of  $\Sigma_1$ at 
$b_1$ yields a regular simplicial complex,
whose maximal triangles $V_2,U_2$
are given by
$$V_2=\conv(o,[2,1,\mathsf F_3],[1,1,\mathsf F_2])
\mbox{ and }
U_2=\conv(o,[1,0,\mathsf F_1],[2,1,\mathsf F_3]).$$
  Let  $\zeta^{(2)}\colon V_2\to
U_2$ be the unique linear extension of the map 
$$o\mapsto o,\,\,[2,1,\mathsf F_3]\mapsto
[2,1,\mathsf F_3],\,\,\,[1,1,\mathsf F_2]\mapsto [1,0,\mathsf F_1].$$
Then    $\zeta^{(2)}$ is a
$\mathbb Z$-homeomorphism  of $V_2$ onto $U_2$, in symbols,
$\zeta^{(2)}\colon V_2\cong_{\mathbb Z} U_2.$ 
Next let $\rho^{(2)}=\zeta^{(2)}\cup \id_{U_2}$.  This is a
$\mathbb Z$-retraction of  $U_1$ onto $U_2$
 acting $\mathbb Z$-homeomorphically
over $V_2$. 
Let  
$$\sigma^{(2)}= \rho^{(2)}\circ \sigma^{(1)}= \rho^{(2)}\circ \rho^{(1)}.$$
This is a
$\mathbb Z$-retraction of  $\I^2$ onto $U_2$
 acting $\mathbb Z$-homeomorphically
over  the following $2^2$ triangles:
$U_2,V_2, (\zeta^{(1)})^{-1}(U_2), (\zeta^{(1)})^{-1}(V_2)$.
By Theorem \ref{Theorem:bella},
 the multiplicity of the retract  $A_2$ of  $\McNtwo$ defined by
$A_2=\range(-\circ \sigma^{(2)})$ is equal to $2^2.$

\bigskip
\noindent{\it Step $2$.} 

Let $\Sigma_2$ be the regular simplicial complex
given by the triangle
$$U_2=\conv(o,[1,0,\mathsf F_2],[2,1,\mathsf F_3])$$
 and all its faces.
In homogeneous coordinates,
let  $b_2=[3,1,\mathsf F_4]$ be the Farey mediant of the edge
$\conv([1,0,\mathsf F_2],[2,1,\mathsf F_3])$
of $U_2$ opposite to the origin $o$.
Then the blow-up of  $\Sigma_2$ at 
$b_2$ yields a regular simplicial complex,
whose maximal triangles $V_3,W_3$
are given by
$$V_3= \conv(o,[3,1,\mathsf F_4],[2,1,\mathsf F_3])
\mbox{ and }
W_3=\conv(o,[1,0,\mathsf F_2], [3,1,\mathsf F_4]).$$

We now let  $[1,0,\mathsf F_3]$
be  the Farey mediant of
$[1,0,\mathsf F_2]$ and $o=[0,0,\mathsf F_1].$
Let  $\mathcal W_3$ be the (regular) simplicial complex
given by  $W_3$ and all its faces.
By  
  blowing-up $\mathcal W_3$
  at  $[1,0,\mathsf F_3]$ we obtain 
a regular triangulation of $W_3$
whose maximal triangles $U_3$ and $T_3$
are given by
$$U_3= \conv(o,[1,0,\mathsf F_3],[3,1,\mathsf F_4])
\mbox{ and }
T_3=\conv([1,0,\mathsf F_3],[1,0,\mathsf F_2], [3,1,\mathsf F_4]).$$
By construction $U_2=V_3\cup W_3=V_3\cup U_3 \cup T_3$. 
 %
%
%\commento{
%L to D: It is ok for me. I wonder if the referee would understand why we don't provide the precise number. But I like very much your picture, so let's keep this option. we can always blow-up $T_3$ and send its farey median to $(0,0)$, it we really wanted. BUt it is fine as is. I really hope we can find a journal that publishes color figures}
%%
%%

Let  $\zeta^{(3)}\colon V_3\to
U_3$ be the unique linear extension of the map 
$$o\mapsto o, \,\,\, [3,1,\mathsf F_4]\mapsto
[3,1,\mathsf F_4],\,\,\,[2,1,\mathsf F_3]\mapsto [1,0,\mathsf F_3].$$
As above,   $\zeta^{(3)}$ is a
$\mathbb Z$-homeomorphism  of $V_3$ onto $U_3$, in symbols,
$\zeta^{(3)}\colon V_3\cong_{\mathbb Z} U_3.$

Let $\lambda^{(3)}\colon T_3 \to U_3$ be the unique linear extension
of the map 
$$[1,0,\mathsf F_3]\mapsto [1,0,\mathsf F_3],
\,\,\,
[3,1,\mathsf F_4]\mapsto [3,1,\mathsf F_4],
\,\,\,[1,0,\mathsf F_2]\mapsto o.$$
Then the map   $\rho^{(3)}=\zeta^{(3)}\cup \lambda^{(3)}\cup \id_{U_3}$  is a
$\mathbb Z$-retraction of  $U_2$ onto $U_3$
 acting $\mathbb Z$-homeomorphically
over $V_3$ and, trivially, over $U_3$.
(Actually,  $\rho^{(3)}$ also 
acts $\mathbb Z$-homeomorphically
over  $T_3$, but for our purposes it is sufficient to
restrict attention to the action of $\rho^{(3)}$ 
over the two triangles $V_3$ and $U_3$.)
The map  
$$\sigma^{(3)}= \rho^{(3)}\circ \sigma^{(3)}
= \rho^{(3)}\circ \rho^{(2)}\circ \rho^{(1)}$$
 is a
$\mathbb Z$-retraction of  $\I^2$ onto $U_3$
 acting $\mathbb Z$-homeomorphically
over  (among others)  the following $2^3$ triangles:
\begin{equation}
\label{equation:foregoing}
U_3,\,\,V_3,\,\, (\zeta^{(2)})^{-1}(U_3),\,\, (\zeta^{(2)})^{-1}(V_3),\,\,\,
(\zeta^{(1)})^{-1}(\mbox{the foregoing 4 triangles}).
\end{equation} 
By Theorem \ref{Theorem:bella}(c), the
 multiplicity of the retract  $A_3$ of  $\McNtwo$ defined by
$A_3=\range(-\circ \sigma^{(3)})$ is 
\commento{L to D: Since we are not showing the other triangles $\mathbb Z$-homoemorphic to $U_3$ I think we should write here $\geq 2^3$ instead of $> 2^3$.}
  $\geq 2^3.$
%  $>2^3.$

\commento{L to D: Weren't we going to eliminate Step 3?}
\bigskip
\noindent{\it Step 3.} 

Let the regular simplex 
$\Sigma_{3}$ consist of the
  triangle  
  $$U_{3}=\conv(o,[1,0,\mathsf F_3],[3,1,\mathsf F_4])$$
   and all its faces.
In homogeneous coordinates,
let  $b_{3}=[4,1,\mathsf F_5]$ be the Farey mediant of the edge
$\conv([1,0,\mathsf F_3],[3,1,\mathsf F_4])$
of $U_3$ opposite to the origin $o$.
Then the blow-up of  $\Sigma_3$ at 
$b_3$ yields a regular simplicial complex,
whose maximal triangles $V_4,W_4$
are given by
$$V_4= \conv(o,[4,1,\mathsf F_5],[3,1,\mathsf F_4])
\mbox{ and }
W_4=\conv(o,[1,0,\mathsf F_3], [4,1,\mathsf F_5]).$$
We now let  $[1,0,\mathsf F_4]$
be  the Farey mediant of
$[1,0,\mathsf F_3]$ and $o=[0,0,\mathsf F_1].$

Let  $\mathcal W_4$ be the (regular) simplicial complex
given by  $W_4$ and all its faces.
By blowing-up $\mathcal W_4$
at  $[1,0,\mathsf F_4]$, we obtain
a regular triangulation of $W_4$
whose maximal triangles $U_4$ and $T_4$
are given by
$$U_4= \conv(o,[1,0,\mathsf F_4],[4,1,\mathsf F_5])
\mbox{ and }
T_4=\conv([1,0,\mathsf F_4],[1,0,\mathsf F_3], [4,1,\mathsf F_5]).$$
Observe that $U_3=V_4\cup W_4=V_4\cup U_4 \cup T_4$. 

Let  $\zeta^{(4)}\colon V_4\to
U_4$ be the unique linear extension of the map 
$$o\mapsto o, \,\,\, [4,1,\mathsf F_5]\mapsto
[4,1,\mathsf F_5],\,\,\,[3,1,\mathsf F_4]\mapsto [1,0,\mathsf F_4].$$
As above,  $\zeta^{(4)}$ is a
$\mathbb Z$-homeomorphism  of $V_4$ onto $U_4$, in symbols,
$\zeta^{(4)}\colon V_4\cong_{\mathbb Z} U_4.$

               \begin{figure} 
    \begin{center}                                     
    \includegraphics[height=8.9cm]{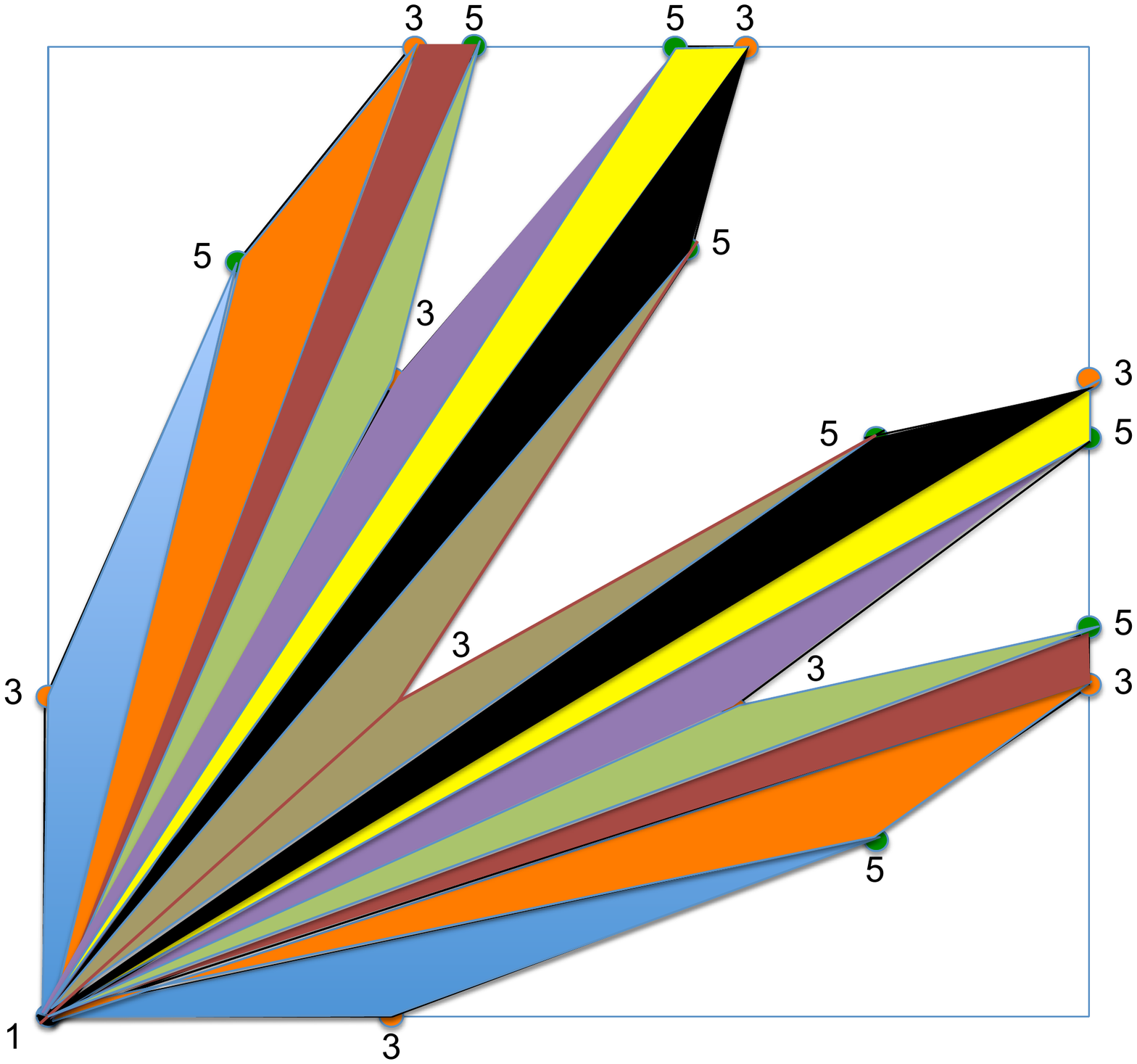}    
    \end{center}                                       
 \caption{\small  Sixteen $\mathbb Z$-homeomorphism
 domains of the map $\sigma^{(4)}$ of Step 3
 in the proof of Theorem \ref{theorem:fibonacci}.
Each $\mathbb Z$-homeomorphism
 domain is a regular triangle whose vertices have
 denominators 1,3,5.}  
    \label{figure:rainbow}                                                 
   \end{figure}

Let $\lambda^{(4)}\colon T_4 \to U_4$ be the unique linear extension
of the map 
$$[1,0,\mathsf F_4]\mapsto [1,0,\mathsf F_4],
\,\,\,
[4,1,\mathsf F_5]\mapsto [4,1,\mathsf F_5],
\,\,\,[1,0,\mathsf F_3]\mapsto o.$$
Then the map   
$$\rho^{(4)}=\zeta^{(4)}\cup \lambda^{(4)}\cup \id_{U_4}$$  is a
$\mathbb Z$-retraction of  $U_3$ onto $U_4$
 acting $\mathbb Z$-homeomorphically
over $V_4$. 
The map  
$$\sigma^{(4)}= \rho^{(4)}\circ \sigma^{(4)}
= \rho^{(4)}\circ \rho^{(3)}\circ \rho^{(2)}\circ  \rho^{(1)}$$
 is a
$\mathbb Z$-retraction of  $\I^2$ onto $U_4$
 acting $\mathbb Z$-homeomorphically
over  the following $2^4$ triangles:
$$
U_4,\,\,V_4,\,\, (\zeta^{(3)})^{-1}(U_4),\,\, (\zeta^{(3)})^{-1}(V_4),\,\,\,
\mbox{etc. etc. etc. unfolding}.
$$
(As in the previous step,  
$\sigma^{(4)}$ acts $\mathbb Z$-homeomorphically
over other triangles, but for our present purposes
it is convenient to restrict attention to these $2^4$ only.
See Figure \ref{figure:rainbow}.) 
By Theorem \ref{Theorem:bella}(c), the
 multiplicity of the retract  $A_4$ of  $\McNtwo$ defined by
$A_4=\range(-\circ \sigma^{(4)})$ is 
$\geq 2^4.$
% $>2^4.$
 
\bigskip
\noindent{\it Step $n-1$,  $(n=5,6,\dots).$}

Inductively let the regular simplex
$\Sigma_{n-1}$ consist of the
  triangle  
  $$U_{n-1}=\conv(o,[1,0,\mathsf F_{n-1}],[n-1,1,\mathsf F_{n}])$$
   and all its faces. 
In homogeneous coordinates,
let  $b_{n-1}=[n,1,\mathsf F_{n+1}]$ be the Farey mediant of the edge
$\conv([1,0,\mathsf F_{n-1}],[n-1,1,\mathsf F_n])$
of $U_{n-1}$ opposite to the origin $o$.
Then the blow-up of  $\Sigma_{n-1}$ at 
$b_{n-1}$ yields a regular simplicial complex,
whose maximal simplexes  $V_n,W_n$
are given by
$$V_n= \conv(o,[n,1,\mathsf F_{n+1}],[n-1,1,\mathsf F_n])
\mbox{ and }
W_n=\conv(o,[1,0,\mathsf F_{n-1}], [n,1,\mathsf F_{n+1}]).$$
%In homogeneous coordinates, $[1,0,\mathsf F_n]$
%is the Farey mediant of
%$[1,0,\mathsf F_n]$ and $o=[0,0,\mathsf F_1].$
Let the regular triangle
$U_n\subseteq W_n$ be given by 
$U_n=\conv(o,[1,0,\mathsf F_{n}], [n,1,\mathsf F_{n+1}]).$
Let  $\zeta^{(n)}\colon V_n\to
U_n$ be the unique linear extension of the map 
$$o\mapsto o, \,\,\, [n,1,\mathsf F_{n+1}]\mapsto
[n,1,\mathsf F_{n+1}],\,\,\,[n-1,1,\mathsf F_n]\mapsto [1,0,\mathsf F_n].$$
The regularity of $V_n$ and $U_n$ ensures that
 $\zeta^{(n)}$ is a
$\mathbb Z$-homeomorphism  of $V_n$ onto $U_n$, in symbols,
$\zeta^{(n)}\colon V_n\cong_{\mathbb Z} U_n.$

For each $j=0,\dots,\mathsf F_{n-2}-1,$ let the triangle
$T_{n,j}$ be defined by
$$
T_{n,j}=\conv([n,1,\mathsf F_{n+1}],[1,0,\mathsf F_{n-1}+j],[1,0,\mathsf F_{n-1}+j+1]).
$$

               \begin{figure} 
    \begin{center}                                     
    \includegraphics[height=8.9cm]{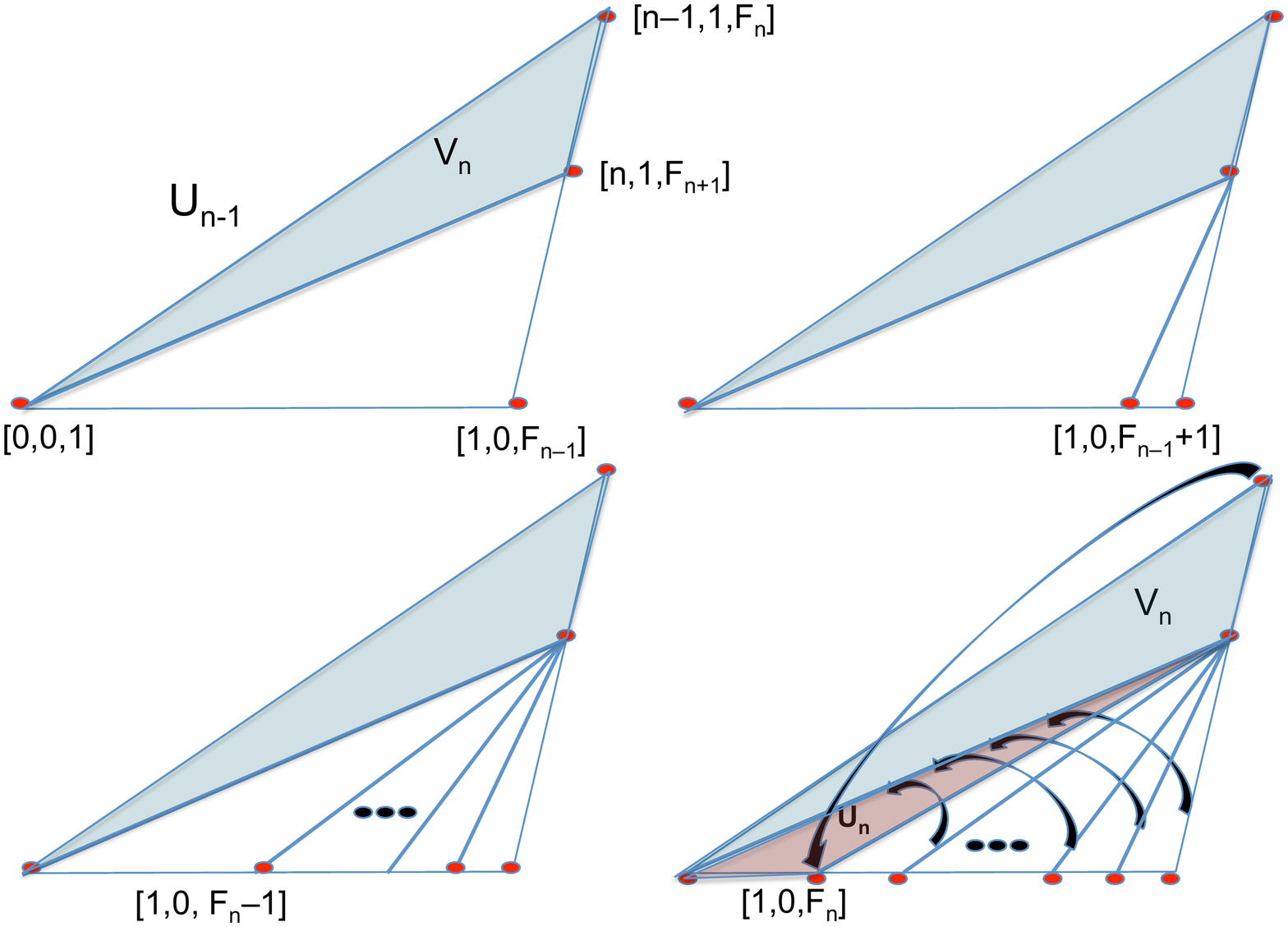}    
    \end{center}                                       
 \caption{\small  The sequence of Farey blow-ups
and retractions in the proof of  Theorem
\ref{theorem:fibonacci}.}  
    \label{figure:howto}                                                 
   \end{figure}

A direct verification shows that every 
 $T_{n,j}$ is regular.
As a matter of fact, the triangle 
  $W_n=\conv(o,[1,0,\mathsf F_{n-1}],[n,1,\mathsf F_{n+1}])$
 is regular; the points  $[1,0,\mathsf F_{n-1}+1]$,
$[1,0,\mathsf F_{n-1}+2],\dots,[1,0,\mathsf F_{n-1}+\mathsf F_{n-2}-1],
[1,0,\mathsf F_{n-1}+\mathsf F_{n-2}]=[1,0,\mathsf F_{n}]$,
are obtained by taking the (always Farey)
 mediant $[1,0,\mathsf F_{n-1}+1]$  of  $o$ and  
 $[1,0,\mathsf F_{n-1}]$, and then taking the mediant 
  $[1,0,\mathsf F_{n-1}+2]$   of  $o$ and  
 $[1,0,\mathsf F_{n-1}+1]$,\dots, 
 and finally taking the mediant  $[1,0,\mathsf F_n]$  
of  $o$ and
 $[1,0,\mathsf F_{n}-1]
 =[1,0,\mathsf F_{n-1}+\mathsf F_{n-2}-1]$. 
Let $\mathcal W_n$ be the regular
  simplicial complex given
by $W_n$ and all its faces.
Then  $U_n$ and the $T_{n,j}$ are
the maximal simplexes of a regular
triangulation of $W_n,$  which is obtained
from $\mathcal W_n$ 
by consecutive  Farey blow-ups as described 
in  Figure \ref{figure:howto}.
Observe that $U_{n-1}=V_n\cup W_n
=V_n\cup U_n \cup \bigcup_j T_{n,j}$.

For each  $j=0,\dots,\mathsf F_{n-2}-1$,  let  
$$
\lambda^{(n)}_j\colon \conv([1,0,\mathsf F_n],
[n,1,\mathsf F_{n+1}],
[1,0,\mathsf F_{n-1}+j])\to U_n
$$
 be the unique linear extension
of the map 
$$[1,0,\mathsf F_n]\mapsto [1,0,\mathsf F_n],
\,\,\,
[n,1,\mathsf F_{n+1}]\mapsto [n,1,\mathsf F_{n+1}],
\,\,\,[1,0,\mathsf F_{n-1}+j]\mapsto o.$$
By \cite[Lemma 3.7]{mun11},
 each  $\lambda^{(n)}_j$ is linear with integer coefficients 
  (i.e.,  $\lambda^{(n)}_j$  is 
 a linear $\mathbb Z$-map)
sending the regular
 triangle $\conv([n,1,\mathsf F_{n+1}],
 [1,0,\mathsf F_n],[1,0,\mathsf F_n-1])$ 
onto $U_n$, and mapping all other  triangles
 $T_j$ onto the segment
$\conv(o,[n,1,\mathsf F_{n+1}])\subseteq U_n.$
The  map   
$$\rho^{(n)}=\zeta^{(n)}\cup \id_{U_3}
\cup \bigcup_{j=0}^{-1+\mathsf F_{n-2}}
\lambda_j^{(n)}$$
  is a
$\mathbb Z$-retraction of  $U_{n-1}$ onto $U_n$
 acting $\mathbb Z$-homeomorphically
over $V_n$. 
The map  
\[\sigma^{(n)}= \rho^{(n)}\circ \sigma^{(n-1)}= \rho^{(n)}\circ \rho^{(n-1)}\circ
\dots\circ  \rho^{(1)}
\]
 is a
$\mathbb Z$-retraction of  $\I^2$ onto $U_n$. 
Generalizing
\eqref{equation:foregoing},  $\sigma^{(n)}$  is a
$\mathbb Z$-ho\-m\-eo\-mor\-phi\-sm   onto  $\,\,U_n\,\,$ 
of each one
of the following $\,\,2^n\,\,$ triangles:
$U_n,\,\,V_n,\,\,$ $ (\zeta^{(n-1)})^{-1}(U_n), \,\,$
$(\zeta^{(n-1)})^{-1}(V_n),\,$
$(\zeta^{(n-2)})^{-1}(\mbox{the foregoing four triangles}),$
$\,\,\dots,\,\, (\zeta^{(2)})^{-1}$(the foregoing  $2^{n-2}$  triangles)\,\,\,
$(\zeta^{(1)})^{-1}$ (the foregoing  $2^{n-1}$  triangles).
Actually, $\sigma^{(n)}$  is a
$\mathbb Z$-ho\-m\-eo\-mor\-phi\-sm  
also of other triangles onto  $\,\,U_n\,\,$,
 but these are irrelevant to our purposes.
By Theorem \ref{Theorem:bella},   the
 multiplicity of the retract  $A_n=\range(-\circ \sigma^{(n)})$ of  $\McNtwo$
 is  
$\geq 2^n$.
%$>2^n$.
 Since the area of 
 $U_n$ is $>0$, by  \eqref{equation:uniform-bound} the
 multiplicity of $A_n$ is finite.

\medskip
Iterating this inductive procedure
we obtain  retracts  $A_m$  of $\McNtwo$ whose maximal ideal
space is a closed domain in $\I^2$,  and whose multiplicity 
and index are
finite and arbitrarily large. 
\end{proof}

\begin{corollary}  Adopt the notation
of \eqref{equation:homogeneous}-\eqref{equation:fibonacci}.
For each  $n=1,2,\dots$ let the triangle
$U_n\subseteq \I^2$ be defined by
$U_n=\conv([0,0,1],[1,0,\mathsf F_n],[n,1,\mathsf F_{n+1}])$.
Then $2^n\leq \iota(\McN(U_n))=\iota(\McNr(U_n))
\in \mathbb Z.$ 
\end{corollary}

%\commento{D to L: actually $\iota(\McN(U_n))=2^n$, but I can't prove it.}

\begin{proof}  
By \cite[Lemma 3.6]{mun11}, the retract  
$A_n=\range(-\circ \sigma^{(n)})$
of  $\McNtwo$  
 in the foregoing theorem
 is isomorphic to $\McN(U_n)$. 
 So  $\iota(\McN(U_n))\geq 2^n.$
 The preservation
 properties of the  $\Gamma$ functor ensure that
 $\iota(\McNr(U_n))=\iota(\McN(U_n))$.
 %
 %
% \commento{D to L: do you have an easy proof of the converse
%inequality ? otherwise we skip it}
 %
 %
 %
\end{proof}

\begin{corollary}
\label{theorem:fibonacci-bis}
 For every  $j=1,2,\dots,$ there is
 retract   $R_j$  of  the free unital $\ell$-group
 $\McNr(\I^2)$ such that  $\iota(R_j)>j,$
and   the maximal spectral space of $R_j$  
is a closed domain in $\I^2$.
\end{corollary}

\medskip

While the index of a finitely generated projective
MV-algebra  arises
by taking the {\it sup} of multiplicities,  taking the 
{\it inf\/} is of little  interest:

\begin{proposition}
\label{proposition:ieri}
Let $A$ be a retract of $\McNn$, say
$A=\range(-\circ\rho)$ for some  $\mathbb Z$-retraction
$\rho$ of  $\I^n$ onto the rational polyhedron  $P$. 
Suppose  $P$   
%(equivalently,  its homeomorphic
%copy  $\mu_A$ given by the maximal spectral space 
%of $A$)
  is a closed domain in $\I^n$.
Then $A$ has an isomorphic
copy  $A'=\range(-\circ\rho')$ where  $\rho'$ 
is a  $\mathbb Z$-retraction of
 $\I^n$ onto $P$ and the multiplicity of $A'$ is
 equal to $1$.
\end{proposition}

\begin{proof}
If the multiplicity of $A$ is 1 we are done.
Otherwise, let $m>1$ be the multiplicity of $A$.
The finiteness of $m$ follows from
Theorem \ref{theorem:finite} since by assumption,
 $P=\cl(\interior(P))$.
By Theorem \ref{Theorem:bella},   there
are exactly  $m$ rational polyhedra
$P=Q_1,Q_2,\dots,Q_m\subseteq \I^n$ such that
for each $i=1,\dots,m,\,\,\,$ $\rho\restrict Q_i$ is
a $\mathbb Z$-homeomorphism of $Q_i$ onto $P.$
Now consider $Q_2$.   By the final part of the proof of  Theorem
\ref{theorem:finite} some connected component,
say $Q$, of the interior of $Q_2$ is disjoint from the
interior of $P.$ Let  
 $\nabla$ be a regular triangulation of $\I^n$
 having the following properties:
 \begin{itemize}
\item[(i)]   $\nabla$
{\it linearizes}  $\rho$ (i.e., $\rho$ is linear over each
simplex of $\nabla$);
\item[(ii)] 
 $\nabla$ has an  
$n$-simplex  $T=\conv(t_0,t_1,\dots,t_n)$ lying in the interior of $Q$,
where 
 $d=\den(t_0)$;

\item[(iii)]   
$\nabla$ also has a vertex 
$t^* \in Q\setminus T$ of denominator  $d$.
\end{itemize}
The existence of  $\nabla$ is ensured by \cite[Proposition 3.2]{mun11}.
Applying  \cite[Lemma 3.7]{mun11} to the
regular simplex  $T$ and to   the
$n+1$ rational points
$t^*,t_1,\dots,t_n$
we obtain a
$\mathbb Z$-map  $\zeta\colon \I^n\to\I^n$ such that
 $\zeta$  is identity over  $\I^n\setminus \interior(T)$,
$\zeta(t_0)=t^*,$  and $\zeta(T)$ is contained in $Q$.  
It follows that   $\zeta\restrict U$ is not one-one.

Thus the composite $\mathbb Z$-map $\rho^{(1)}=\rho\circ\zeta$, while
being a $\mathbb Z$-retraction of $\I^n$ onto $P$, does not
act $\mathbb Z$-homeomorphically over $Q_2$. 
By Theorem \ref{Theorem:bella},
the multiplicity of the retract  $A_1= \range(-\circ\rho^{(1)})$
is equal to  $m-1$. Proceeding in this way we can find
a  $\mathbb Z$-retraction $\rho^{(m-1)}$ of $\I^n$ onto $P$ such that
the multiplicity of the  retract  $A_{m-1}=\range(-\circ\rho^{(m-1)})$
is equal to 1.  Since all $\mathbb Z$-retractions
$\rho^{(1)},\dots,\rho^{(m-1)}$ are onto the same rational
polyhedron $P$,  $A_{m-1}$ is isomorphic to
$A$.  Now set ${A'=A_{m-1}}$. \end{proof}

\section{Comparing retracts of free MV-algebras and unital $\ell$-groups}
%%%%%%%%%%%%%%%%%%%%%%%

Two sets  $A,B $ are said to be {\it comparable} if either $A\subseteq B$ or 
$B\subseteq A$.

\begin{proposition}
\label{proposition:counting}
Any two  $\mathbb Z$-homeomorphic comparable rational polyhedra
$P,Q\subseteq \I^n$ are equal.
However, two isomorphic comparable finitely presented
subalgebras of $\McNn$ need not be equal.
\end{proposition} 

\begin{proof} 
Suppose  $P\subseteq Q$
and $P\not=Q.$  Then
for some suitably large  integer $d$ the number of points
of denominator $d$ in $P$ is strictly less than in $Q$.
By \cite[Proposition 3.15]{mun11}, 
 $P$ and $Q$ are not $\mathbb Z$-homeomorphic.
For the second statement,    the subalgebra of $\McN(\I)$
generated by $x\oplus x$  is isomorphic to
$\McN(\I)$  but is not equal to $\McN(\I)$.
\end{proof}

\begin{theorem}
\label{theorem:criterion}
Retracts $A,B$ of $\McNn$ are equal iff
they are comparable and isomorphic.
\end{theorem}

\begin{proof} For the nontrivial direction, assume $A\cong B$
and $A\subseteq B$, with the intent of proving $A=B.$
For suitable McNaughton functions  $\sigma_1,\ldots,\sigma_n$ and
$\tau_1,\ldots,\tau_n$ with $\sigma\circ\sigma=\sigma$ and 
$\tau\circ\tau=\tau$
 we can write $A=\gen(\sigma_1,\ldots,\sigma_n)$
and $B=\gen(\tau_1,\ldots,\tau_n)$.
The restriction to $B$ of the retraction $-\circ \sigma
\colon \McNn\to A$ is a retraction of $B$ onto $A$, and we have
a commutative diagram

$$
\begin{CD}
B@>(-\circ\sigma)\restrict B>> A\\
      @A {\rm id} AA  @A{\rm id}AA \\
B@<<{\rm inclusion}<A
\end{CD}
$$

\bigskip

\noindent
Dually,   \cite{marspa}, 
 we get the commutative diagram
\begin{equation}
\label{equation:observe-bis}
\begin{CD}
\mathsf{R}_\tau@<\epsilon<<\mathsf{R}_\sigma\\
      @V {\rm id} VV  @V{\rm id}VV \\
\mathsf{R}_\tau@>\delta>>\mathsf{R}_\sigma\\
\end{CD}
\end{equation}

\bigskip
\noindent
Since  $A\subseteq B$, from 
\cite[Theorem 3.2(ii)]{cab-forum} 
it follows that  $\delta$  is
 onto $\mathsf{R}_\sigma$.
By  \cite[Theorem 3.5]{cab-forum},   $\epsilon$ is one-one and
preserves  denominators.
Since   $A\cong B$,  by   the (Cantor-Bernstein)
theorem  \cite[Theorem 3.7]{cab-forum}, 
$\epsilon$  is a 
$\mathbb Z$-homeomorphism of  $ \mathsf{R}_\sigma$
onto $\mathsf{R}_\tau$, 
 $$
 \epsilon\colon  \mathsf{R}_\sigma \cong_{\mathbb Z}  \mathsf{R}_\tau. 
 $$
 Now from \eqref{equation:observe-bis} it follows that
 $\delta$ and $\epsilon$ are inverses of each other, whence
 $$
  \delta\colon  \mathsf{R}_\tau \cong_{\mathbb Z}  \mathsf{R}_\sigma.
 $$ 
Therefore, 
 the inclusion map of $A$ into $B$
  (which is the dual of  $\delta$) 
  is surjective,  and 
$A=B.$  
\end{proof}

The following result is a special case of
 \cite[Theorem 4.6]{cabmun-jlc}. 
 We have included it here because of its simple proof.

\begin{proposition}
Let $A$ be a
 {\rm separating} retract of $\McNn$, in the sense
 that for all distinct $x,y\in \I^n$ there is $f\in A$ with
 $f(x)\not=f(y).$
 Then $A$ coincides with $\McNn.$
\end{proposition}

\begin{proof}
By hypothesis we have a retraction  $\epsilon$ of $\McNn$
onto  $A.$
Letting  $\sigma_i=\epsilon(\pi_i),\,\,\,(i=1,\dots,n),\,$
and recalling the notational stipulations in the
introductory part of Section \ref{section:introductory}, the 
  retraction  $\epsilon$  determines the
$\mathbb Z$-retraction  $\zret_\epsilon=
\sigma=(\sigma_1,\dots,\sigma_n)\colon \I^n\to \I^n$,
 and we can write  $A=\gen(\sigma_1,
\ldots,\sigma_n)$.   
 If $\mathsf{R}_\sigma=\I^n$
we are done, and $\sigma$ is  identity on $\I^n$. 
If $\mathsf{R}_\sigma$ is strictly contained in $\I^n$ then some
rational point $r \in \I^n$ of sufficiently high denominator
does not belong to $\mathsf{R}_\sigma$. 
Since  $\sigma(r)$ lies in $\mathsf{R}_\sigma$, necessarily
  $\sigma(r)\not=r.$ Since
$\sigma(r)=\sigma(\sigma(r))$ then for each $f\in A$
we must have  $f(r)=f(\sigma(r))$, because
$f$ has the form  $g\circ \sigma$ for some  $g\in \McNn.$ 
We conclude that $A$ is not a separating subalgebra of $\McNn.$
\end{proof}

%\begin{definition} 
%Two retractions of  $\McNn$  (resp., two
%$\mathbb Z$-retractions of $\I^n$) are said to be
%{\it confluent} if they have the same range.
%\end{definition}

%\begin{definition}
%Two subalgebras of $\McNn$  (resp., two
%rational polyhedra in  $\I^n$) are said to be
%{\it incomparable} if none is a subset of the other.
%\end{definition}

\begin{proposition} 
\label{proposition:incomparable}
For  any two  $\mathbb Z$-retractions  $\sigma\not=\tau$
of $\I^n$ with equal range,  the retracts $A_\sigma = \gen(\sigma_1,\ldots,\sigma_n)$ 
and $A_\tau = \gen(\tau_1,\ldots,\tau_n)$ are  (isomorphic and)
 incomparable. 
\end{proposition}

%\commento{D to L: Here, perhaps, the reader
%will appreciate seeing   your example
%where  $A_\sigma\cap A_\tau$ strictly contains $\{0,1\}$.}

\begin{proof} 
Isomorphism immediately follows from 
 \cite[Corollary 3.10]{mun11}. Concerning
incomparability, the
 solutions of the equation
$$
\xi\circ\sigma=\sigma
$$
in the unknown  $\xi=(\xi_1,\ldots,\xi_n)\colon \I^n\to \I^n,
\,\,\,\xi_i\in \McNn,$ are
precisely those elements of  $ (\McNn)^n$ which act identically 
on the range  $\mathsf R_\sigma$.
If  $A_\tau$ is a subalgebra of
$A_\sigma$   then for some  $\chi
=(\chi_1,\dots,\chi_n)$  with $\chi_i\in\McNn$ we have
$\chi\circ\sigma=\tau$, because  $\{\sigma_1,\dots,
\sigma_n\}$ is a generating set of   $A_\sigma.$
Over the polyhedron 
 $\mathsf{R}_\sigma=\mathsf{R}_\tau$, the function $\chi$  must act
identically, because so do $\sigma$ and $\tau.$ 
Similarly, for each  $y\in \I\setminus \mathsf{R}_\sigma$ 
the point
$\sigma(y)$ lies  in $\mathsf{R}_\sigma$. We have proved the identity 
$(\chi\circ\sigma)(y)=\sigma(y)$ for all $y\in \I^n$.
From  $\chi\circ\sigma=\sigma$ and  $\chi\circ\sigma=\tau$
we get  $\sigma=\tau.$ 
\end{proof}

\bigskip

 \noindent
For any fixed  
   $\mathbb Z$-retract $P\subseteq \I^n$  the set
 $\Omega_P$  of MV-algebras  is defined by 
$$
\Omega_P=\{\gen(\sigma_1,\ldots,\sigma_n)\mid \sigma \mbox{ any possible 
$\mathbb Z$-retraction of $\I^n$ onto $P$} \}.
$$
By duality, any two algebras in $\Omega_P$ are isomorphic.  

%With the exception of the case  $P=\I^n$,  there are
%(countably) infinitely many algebras in $\Omega_P$.
%That's almost all we can say about the class $\Omega_P$.
% 

\begin{proposition}  In general, the intersection of two  
MV-algebras in
 $\Omega_P$  need not be in  $\Omega_P$.
 The smallest MV-algebra containing two MV-algebras in 
  $\Omega_P$  need not be in  $\Omega_P$. 
\end{proposition}

\begin{proof}  For both statements we have examples already for
 $n=1$.

For the first statement, let  $\sigma=\pi_1\wedge \neg \pi_1$.
Then  
the map $f\mapsto f\circ \sigma$ amounts to  taking the mirror image
of the first half of $f$. 
  Let now  $\tau\colon \I\to \I$
act identically  on  the interval $[0,1/2]$, then descend to 0
with slope  $-3$, and finally vanish over  $[2/3,1]$.  
 All functions   $f\in A_\sigma\cap A_\tau$
are symmetric around the axis  $y=1/2$, and are constant
over the interval $[2/3,1]$, so they are also constant
over the interval $[0,1/3].$  As a consequence,
$A_\sigma\cap A_\tau$ does not have a maximal
quotient isomorphic to 
$\Gamma(\mathbb Z\frac{1}{3},1).$
By \cite[Lemma 3.6]{mun11}, every MV-algebra
  $A$  in $\Omega_{[0,1/2]}$ is isomorphic to
$\McN([0,1/2])$. So, in particular,  $A$ has a maximal
quotient
isomorphic to 
$\Gamma(\mathbb Z\frac{1}{3},1).$
So $A_\sigma\cap A_\tau
\notin \Omega_{[0,1/2]}.$

For the second statement take  
two different  $\mathbb Z$-retractions
$\sigma,\tau$ of $\I$ onto the the same  range 
$[0,q] \subseteq \I$.  The interval $[0,q]$
is a $\mathbb Z$-retract of $\I$.
Every  MV-algebra  in $\Omega_{[0,q]}$ is
 isomorphic to $A_\sigma$ and hence it is projective.
By duality we
can write   $A_\sigma\cong A_\tau\in \Omega_{[0,q]}.$
Now for definiteness assume both $\sigma$ and $\tau$
to have exactly three linear
pieces. 
We {\it claim} that
  the range $R$ of the map
$(\sigma,\tau)\colon \I\mapsto \I^2$ is not simply connected.
As a matter of fact, let us proceed along the trajectory
$t\in [0,q]\mapsto (\sigma(t),\tau(t))
\in \I^2$ starting from the $(0,0)$ at time  $t=0$.
Then we go up
along the diagonal $x_1=x_2$ of $\I^2$ until, at time  $t=q$,
we reach 
the  point  $(q,q)$;   we  then go down until we reach,
say,  the $x$-axis, and finally move leftward until we
reach the origin, at time $t=1.$
The resulting piecewise linear curve  $R=\range(\sigma,\tau)$ 
is the perimeter of a quadrangle, whence it is not simply connected.
Our claim is settled.

Let  $\gen(\sigma,\tau)$ denote the subalgebra of
$\McN(\I)$ generated by $\sigma$ and $\tau$.
This is the smallest MV-algebra containing  $A_\sigma\cup A_\tau.$
By \cite[Lemma 3.6]{mun11} we  have 
the isomorphism  $\gen(\sigma,\tau)\cong \McN(R)$.
By \cite[Corollary 4.18]{mun11}, the maximal spectrum
of  $\gen(\sigma,\tau)$ 
is homeomorphic to $R$, so it is 
not simply connected, and $\gen(\sigma,\tau)$  is not projective.
We conclude that $\gen(\sigma,\tau)\notin \Omega_{[0,q]}.$ 
\end{proof}

%%%%%%%%%%%%%%%%
\section{Decision problems for projective algebras}
%%%%%%%%%%%%%%%%%%%%

Unless otherwise specified, all MV-terms in this section
are  in the same  variables $X_1,\ldots,X_n.$
We use the adjective ``decidable'' (resp., ``computable'') as an abbreviation
of ``Turing decidable'' (resp., ``Turing computable'').

\begin{proposition}
\label{proposition:idempotence}
The following problem is decidable:

\smallskip
\noindent
${\mathsf{INSTANCE}:}$  MV-terms  $t_1,\ldots,t_n$.

\smallskip
\noindent
${\mathsf{QUESTION}:}$
Is the map  $\hat t=(\hat t_1,\ldots,\hat t_n)$ a
$\mathbb Z$-retraction of $\I^n$? 
\end{proposition}

\begin{proof} Checking the idempotency property $\hat t\circ\hat t=\hat t$ amounts to deciding whether   the MV-term $t_i\leftrightarrow t_i\circ t$ is a tautology
in infinite-valued \luk\ logic  $(i=1,\ldots,n)$. The latter problem is decidable,
\cite[Corollary 4.5.3]{cigdotmun}.
\end{proof}

The foregoing  innocent looking result should be contrasted with the
following:

\begin{proposition}
When a rational polyhedron   $R\subseteq \I^n$  is presented as a union of
rational simplexes in $\I^n$, or even by a rational triangulation, checking whether
$R$ is a $\mathbb Z$-retract is not
a decidable problem.
\end{proposition}

\begin{proof} 
As is well known, 
 $R$ is contractible iff it is a retract of $\I^n$, \cite[Proposition 5.1]{cabmun}. 
 Using both  directions of the characterization
 theorems
 of  $\mathbb Z$-retracts,
  (respectively in  \cite{cabmun} and \cite{cab})
 it follows that 
 $R$ is a $\mathbb Z$-retract  iff it is contractible and
 satisfies the following two conditions:
 \begin{itemize}
 \item[(i)]
 $R$ has a nonempty intersection with the set of vertices of
 $\I^n$;
  \item[(ii)] $R$ has a {\it strongly regular} triangulation i.e., \cite[Definition 4.1]{cabmun}
a  regular triangulation $\Delta$ such that the greatest common divisor
  of the vertices of each maximal simplex of $\Delta$ is equal to 1.
\end{itemize}
Property  (i) is trivially decidable.
Also   (ii) is decidable, because  it is equivalent to
to the strong regularity of
  every regular triangulation of $R$. 
  
By way of contradiction, assume the $\mathbb Z$-retract  problem 
is decidable.  Then
we can decide the contractibility of rational polyhedra
 in $\I^n$, whence the contractibility  of rational polyhedra in
 $\mathbb R^n$ would be a decidable problem. This
 contradicts  \cite[p.242]{sti}.
 \end{proof}

\begin{proposition}
\label{proposition:closed-domain}
The following problem is decidable:

\smallskip
\noindent
${\mathsf{INSTANCE}:}$  MV-terms  $t_1,\ldots,t_n$ such that
the map  
$\hat t \colon \I^n \to \I^n$ is 
idempotent  (a decidable hypothesis,
  by Proposition \ref{proposition:idempotence}).
 Let\/  $\mu_A$ 
denote  the maximal spectrum   of the MV-algebra
$A\subseteq \McNn$ generated by  \, $\hat t_1,\ldots,\hat t_n$.

\smallskip
\noindent
${\mathsf{QUESTION}:}$
Is   $\mu_A$ homeomorphic to a closed domain in $\I^n$?
\end{proposition}

\begin{proof} 
By \cite[Corollary 4.18]{mun11} there is a 
  homeomorphism of  $\mu_A$ onto 
the set $E=\{x\in \I^n\mid x=\hat t(x)\} = \mathsf{R}_{\hat t}. $
 The rational polyhedron $E$ can be
 computed from the input MV-terms
$t_1,\ldots,t_n$.  By \cite[Lemma 18.1]{mun11}, 
a (regular) triangulation $\nabla$ of $E$ can be
 computed.  Then
$\mu_A\cong E$ is a closed domain iff all maximal
simplexes of $\nabla$  are $n$-dimensional. This
property  is
decidable.    
\end{proof}

\begin{theorem}
\label{theorem:computable-multiplicity}
Let $\sigma=(\hat s_1,\ldots,\hat s_n)$ be 
the $\mathbb Z$-retraction of $\I^n$ determined by 
the MV-terms $s_1,\ldots,s_n$.
%, and $-\circ\sigma$ its corresponding retraction  of $\McNn.$  
Let  $P=\mathsf R_\sigma$ be the range of $\sigma$, and
  $A=\gen(\hat s_1,\ldots,\hat s_n)$ be the retract of $\McNn$
  associated to $\sigma$.
If  $P$ is a closed domain,  the multiplicity $\mathsf{r}(A)$  is
computable from the input terms  $s_1,\ldots,s_n.$
\end{theorem}

%\commento{L to D: I have changed some 
%occurrences of $\sigma_1,\ldots,\sigma_n$ by $\hat s_1,\ldots,\hat s_n$.}

\begin{proof} Given the input terms $s_1,\ldots,s_n$
the idempotency of $\sigma$ is
decidable   by Proposition \ref{proposition:idempotence}, and so is
the hypothesis that $P$ is a closed domain,
by Proposition \ref{proposition:closed-domain}.
Let us write $P=\mathsf{R}_\sigma$.
For any $\mathbb Z$-retraction $\tau=(\tau_1,\ldots,\tau_n)$ of $\I^n$ we have 
$\gen(\tau_1,\ldots,\tau_n)=\gen(\hat s_1,\ldots,\hat s_n)$ iff
$\sigma\restrict \mathsf{R}_\tau$ is a $\mathbb Z$-homeomorphism of 
$\mathsf{R}_\tau$ onto $\mathsf{R}_\sigma$. This  is proved in 
Theorem \ref{Theorem:bella}.
Let $\Delta$ be a regular triangulation of $\I^n$ 
that   linearizes  $\sigma$,
(i.e.,  $\sigma$ is linear over each simplex of $\Delta.$)
By \cite[Lemma 18.1]{mun11},
$\Delta$   is   computable from the input MV-terms $s_i.$
Let the subcomplex $\nabla\subseteq \Delta$ of simplexes be defined by
$$
\nabla=\{S\in\Delta\mid \sigma\restrict S
\mbox{ is a ${\mathbb Z}$-homeomorphism of $S$ onto } \sigma(S)\}.
$$
Also  $\nabla$ is   computable from the input MV-terms $s_i.$ 
Let   $|\nabla|$ denote the support of $\nabla$,
$$
|\nabla| = \bigcup\{S\mid S\in \nabla\}.
$$

\medskip
\noindent{\it Claim 1.} Suppose the rational polyhedron $Q\subseteq \I^n$ satisfies
$\sigma\restrict Q\colon Q\cong_{\mathbb Z} P.$  Then $Q\subseteq |\nabla|.$

As a matter of fact, since $P$ is a closed domain in $\I^n$ then so is $Q$.
Fix $x\in\interior(Q)$ together with an $n$-simplex  $S\in \Delta$ such that $x\in S.$
There is a rational simplex  $T$ satisfying  $T\subseteq Q\cap S.$ From
$\sigma\restrict Q\colon Q\cong_{\mathbb Z} P$ we get 
$\sigma\restrict T\colon T\cong_{\mathbb Z} \sigma(T)$.
Since $T$ and $S$ are $n$-simplexes and $\sigma$ is linear over $S$
(because $S\in \Delta$ and $\Delta$ linearizes $\sigma$) then
$\sigma\restrict S\colon S\cong_{\mathbb Z} \sigma(S)$.
We have thus shown  that  $\interior(Q)$ is contained in $|\nabla|$.
Since $Q$ is a closed domain and $|\nabla|$ is closed then  $Q$ is contained
in $|\nabla|$, and our claim is settled.

\medskip
We now strengthen Claim 1 as follows:

\medskip
\noindent{\it Claim 2.} Suppose the rational polyhedron $Q\subseteq \I^n$ satisfies
$\sigma\restrict Q\colon Q\cong_{\mathbb Z} P.$  Then $Q= \bigcup\{S\in \nabla
\mid S\subseteq Q\}.$

To prove this claim, again fix  $x\in\interior(Q)$. By Claim 1 there is
$S\in \nabla$ with $x\in S$. By way of contradiction suppose 
$S$ is not contained in $Q$. Then by Claim 2 in 
Theorem \ref{theorem:finite}
(using the connectedness of  $\interior(S)$)
there is $y\in \interior(S)$
satisfying $y\in Q\setminus \interior(Q).$ 
From
$\sigma\restrict Q\colon Q\cong_{\mathbb Z} P$ 
it follows that  $\sigma(y)\in P\setminus
\interior(P).$
From 
\commento{L to D: Missing $\sigma$ added.}
$\sigma\restrict S\colon S\cong_{\mathbb Z} \sigma(S)$ 
%$\sigma\restrict S\colon S\cong_{\mathbb Z} S$ 
it follows that 
$\sigma(y)\in   \interior(\sigma(S))\subseteq \interior(P).$
This contradiction settles Claim 2.

\medskip
To conclude the proof, for each subset  $\mathcal S$ of $\nabla$
only consisting of $n$-dimensional simplexes,  it is decidable  whether  
$\sigma\restrict \bigcup \mathcal S$ is
a ${\mathbb Z}$-homeomorphism of 
$ \bigcup \mathcal S$ onto  $P$.  Injectivity is equivalent to the following
property:
For any two distinct $k$
simplexes $V,W\in\mathcal S$, 
from  $\relint(V)\cap\relint(W)=\emptyset$
it must follow that $\sigma(\relint(V))\cap\sigma(\relint(W))=\emptyset$.
This amounts to a routine linear algebra problem involving intersections
of rational hyperplanes in $\mathbb R^n,$  once  $V$ and $W$ are
presented as intersections of rational hyperplanes---in an effective way
as in \cite[Lemma 18.1]{mun11}.
Once the injectivity of $\sigma\restrict \bigcup \mathcal S$ has been
verified, we check 
surjectivity by
computing  the $n$-dimensional
Lebesgue measure $\lambda$ of the union of all $n$-dimensional
simplexes in $\mathcal S$.  This is computable
because  $\Delta$ is a rational (actually, a regular) triangulation.
We finally check that  $\lambda$  is equal
to the measure of  $\bigcup\{\sigma(T)\mid T \in \mathcal S\}$. This, too,
is computable, once  the set $\bigcup\{\sigma(T)\mid T \in \mathcal S\}$ 
has been equipped with a regular
triangulation.  In this way, some Turing machine can compute the
set $\Lambda= \mathcal S_1,\dots, \mathcal S_w$ of all 
subsets  $\mathcal S$ of $\nabla$ such that
$\sigma\restrict \bigcup \mathcal S$ is
a ${\mathbb Z}$-homeomorphism of 
$ \bigcup \mathcal S$ onto  $P$.  By 
Theorem \ref{Theorem:bella}, the number   of
elements in $\Lambda$  coincides with the multiplicity of $A,$
\,\,\,$w=\mathsf{r}(A).$   
\end{proof}

\begin{proposition}
\label{proposition:connected-interior}
The following problem is decidable:

\smallskip
\noindent
${\mathsf{INSTANCE}:}$  
MV-terms  $t_1,\ldots,t_n$ such that
the map   $\hat t \colon \I^n \to \I^n$ is 
idempotent, and the maximal spectral space
$\mu_A$ of  $A=\gen(\hat t_1,\ldots, \hat t_n)$
 is homeomorphic to
 a closed domain (both conditions
 being  decidable, respectively
  by Proposition \ref{proposition:idempotence}
 and \ref{proposition:closed-domain}).
 
\smallskip
\noindent
${\mathsf{QUESTION}:}$ Let  $\interior(\mu_A)$ denote the interior of $\mu_A$.  Is  
\, $\interior(\mu_A)$   connected?
\end{proposition}

\begin{proof} Again replace  $\mu_A$ by
 its homeomorphic copy given by the rational
 polyhedron
$E=\{x\in \I^n\mid x=\hat t(x)\}
=\mathsf{R}_{\hat t}$. 
Compute a   rational triangulation  $\Delta$ of $E$.
Verify the closed domain hypothesis by checking that
all maximal simplexes of $\Delta$  are $n$-dimensional.
Call $\Delta^{(n)}$ the collection of all these $n$-simplexes,
ordered lexicographically.
Inductively, letting $X_k$ be the set of  the first $k$
simplexes of $\Delta^{(n)}$, add to  $X_k$  the first simplex of $\Delta^{(n)}$ which shares
a facet with some  simplex of $X_k$.   Denote by
$X_{k+1}$   the new set of $n$-simplexes thus obtained.
 Note that
$X_{k+1}$ has a connected interior if so does $X_{k}$.
After  $u$ steps no more  
$n$-simplexes can be added to $X_u$.  
Then check that
$X_u$ equals $\Delta^{(n)}$.    
\end{proof}

%\commento{D to D and L: Give   interesting  
%computations of the multiplicity (and of the index)
% of retracts of $\McNn$?}

\begin{proposition}
\label{proposition:confluence}
The  following problem is decidable:

\smallskip
\noindent
${\mathsf{INSTANCE}:}$  MV-terms  
$s_1,\dots,s_n$ and $t_1,\ldots,t_n$ providing  $\mathbb Z$-retractions
$\hat s, \hat t$  of
$\I^n$  (a decidable hypothesis,
  by Proposition \ref{proposition:idempotence}).

\smallskip
\noindent
${\mathsf{QUESTION}:}$  Do these two  
$\mathbb Z$-retractions have the same range?
\end{proposition}

\begin{proof} The ranges of $\hat s$  and $\hat t$
are computable from the input terms
$s_1,\dots,s_n$ and  
$t_1,\ldots,t_n$.
 It is decidable whether the two rational polyhedra
$\mathsf{R}_{\hat s}$ and $\mathsf{R}_{\hat t}$  coincide,
\cite[Corollary 18.4]{mun11}.
\end{proof}

%\begin{problem}  Prove or refute the decidability of the
%following problem:
%
%\smallskip
%\noindent
%${\mathsf{INSTANCE}:}$  MV-terms  $t_1,\ldots,t_n$,
%$\upsilon_1,\dots,\upsilon_n$  each  $t_i$ and $\upsilon_j$ 
%in the same variables
%$X_1,\ldots,X_n$, determining $\mathbb Z$-retractions of
%$\I^n$ into  $\I^n$, which is a decidable
%condition by Proposition \ref{proposition:idempotence}.

%\smallskip
%\noindent
%${\mathsf{QUESTION}:}$  Are the two  
%$\mathbb Z$-retractions  $\sigma$ and $t$ conjugate?
%
%\end{problem}

\smallskip
\begin{proposition}
\label{proposition:biconfluent}
The  following problem is decidable:

\medskip
\noindent
${\mathsf{INSTANCE}:}  \,\,$MV-terms  $s_1,\ldots,s_n$, and
$t_1,\dots,t_n$   
yielding  $\mathbb Z$-retractions  $\hat{s}$
and $\hat{t}$ of
$\I^n$ with the same range, (both assumption
being decidable, 
by Propositions \ref{proposition:idempotence} 
and \ref{proposition:confluence}).

\smallskip
\noindent
${\mathsf{QUESTION}:}$  Does   the  MV-algebra   
 generated by $ \hat s_1,\ldots,\hat s_n$
coincide with  the MV-algebra  generated by 
$ \hat t_1,\ldots,\hat t_n$?
\end{proposition}

\begin{proof}
By Proposition \ref{proposition:incomparable}, the answer is
positive answer iff 
   $\hat{s}=\hat{t}$. This in turn is equivalent to
checking whether the MV-term 
$s_i\leftrightarrow t_i$ is a tautology
 for all $i=1,\ldots,n$,  which is
a decidable problem, \cite[Corollary 4.5.3]{cigdotmun}. 
\end{proof}

Dropping the  hypothesis  that $\hat s$ and $\hat t$
have the same range, the problem remains decidable,
yet  with a much subtler proof:

\smallskip
\begin{theorem}
The  following problem is decidable:

\medskip
\noindent
${\mathsf{INSTANCE}:}$
  MV-terms  $s_1,\ldots,s_n$, and
$t_1,\dots,t_n$   
determining $\mathbb Z$-retractions
of  $\I^n$
$\hat{s}=(\hat s_1,\ldots,\hat s_n)$
and $\hat{t}=( \hat t_1,\ldots,\hat t_n)$, (a 
decidable condition, by
Proposition  \ref{proposition:idempotence}).

\medskip
\noindent
${\mathsf{QUESTION}:}$  Does   the  MV-algebra   
 $A$ generated by $ \hat s_1,\ldots,\hat s_n$
coincide with  the MV-algebra $B$  generated by 
$ \hat t_1,\ldots,\hat t_n$?
\end{theorem}

\begin{proof} 
Let $P=\mathsf R_{\hat s}$  be the range of $\hat s$ and 
$Q=\mathsf R_{\hat t}$ be the range of $\hat t$.
If $P$ coincides with $Q$ (a decidable condition, by
Proposition \ref{proposition:confluence}) then Proposition
\ref{proposition:biconfluent} shows that the problem
$A=B$ is decidable.  So it is sufficient to argue in case  $P\not=Q.$
We have 
\begin{equation}
\label{equation:iff}
\mbox{$A=B$ iff  
$\hat{s}\restrict Q$ is a $\mathbb Z$-homeomorphism 
of $Q$ onto $P$, and 
$\hat{t}=(\hat{s}\restrict Q)^{-1}\circ\hat{s}$.} 
\end{equation}

The $(\Rightarrow)$-direction is proved in
Theorem \ref{Theorem:bella}.
For the $(\Leftarrow)$-direction,  the hypothesis shows that
$(\hat{s}\restrict Q)\circ\hat{t}=\hat{s}$, whence
$A=\gen(\hat{s}_1,\dots,\hat{s}_n)=\gen(\rho_1,\dots,\rho_n)=B.$

Next, in order to check the right-hand side of \eqref{equation:iff} we proceed
as follows:

\begin{itemize}
\item[(i)]  Using the effective procedure
of  \cite[Corollary 2.9]{mun11}, we
compute a regular triangulation $\Lambda$
  of $Q$ such that  $\hat{s}\restrict Q$ is linear
  over each simplex of $\Lambda$. 
  In the light of the characterization of $\mathbb Z$-homeomorphisms,
  \cite[Proposition 3.15]{mun11}, we  then
 check whether 
 \begin{itemize}
 \item[---]
  each maximal simplex  $M$ of  $\Lambda$ is sent
by $\hat{s}$ onto a  regular simplex $\Lambda(M)\subseteq
P$  with preservation of the denominators of the vertices of $M$;
 \item[---]
  the relative interiors
of any two distinct simplexes $M',M''$ of  $\Lambda$
are sent to disjoint simplexes $\Lambda(M'),\Lambda(M'')$;

 \item[---] the $i$-dimensional rational measure 
\cite{mun-cpc} of  $\Lambda(Q)$ coincides with
the  $i$-dimensional rational measure of $Q$,
for each  $i=0,1,\dots,n.$
 \end{itemize}

\item[(ii)]  The  three conditions above
are necessary and sufficient for 
$\hat{s}\restrict Q$ to be a $\mathbb Z$-homeomorphism 
of $Q$ onto $P$.

\item[(iii)]
Using the    extension
argument, \cite[Theorem 5.8(ii)]{mun11} it is easy to 
compute  MV-terms  $r_1,\dots,r_n$
such that the $\mathbb Z$-map $\hat r=(\hat r_1,\dots,\hat r_n)$ coincides
with 
$(\hat{s}\restrict Q)^{-1}$ over  $P$.

\item[(iv)] The verification of the identity
$\hat{t}=(\hat{s}\restrict Q)^{-1}\circ\hat{s}$ now amounts
to checking whether the MV-term
$t_i\leftrightarrow r_i\circ(s_1,\dots,s_n)$ is a tautology in \luk\ logic for
all $i=1,\dots,n$, which, as we have seen, is decidable.
\end{itemize}
The proof is complete.
\end{proof}
%
%******
% 
%\item  Using \cite[Corollary 2.9]{mun11}  
% we first compute a regular triangulation $\Delta$ of 
%$Q$ such   that its image  $\hat s(\Delta)=\{\hat s(T)
%\mid T \in \Delta\}$   {\bf linearizes}  $\hat t\restrict P$.
%
%\item Next we similarly  compute a
%regular refinement  $\nabla$ of  $\hat s(\Delta)$ that
%linearizes
%$\hat t$ on $P$,  whence its image  $\nabla'=\hat s^{-1}(\nabla)$ will linearize
%$\hat s\restrict Q.$
%
%\item   Now a vertex-by-vertex check
% from $\nabla'$  and $\hat s(\nabla')$ we can
%effectively check whether the composite map
%$\hat s\circ \hat t$ acts identically on each vertex of the triangulation.
%%
%%
%\commento{D to D:which triangulation?}
%%
%%
%%

\medskip
Replacing identity by isomorphism in the foregoing theorem
 we have an open problem:

\smallskip
\begin{problem} 
The  following problem is open:

\medskip
\noindent
${\mathsf{INSTANCE}:}$  
MV-terms  $s_1,\ldots,s_n$, and
$t_1,\dots,t_n$   
yielding  $\mathbb Z$-retractions
$\hat{s}$
and $\hat{t}$  of
$\I^n$,  (a decidable condition,
  by Proposition \ref{proposition:idempotence}).

\medskip
\noindent
${\mathsf{QUESTION}:}$  Is the subalgebra     of
$\McNn$ generated by   
  $\hat s_1,\ldots,\hat s_n$
isomorphic to  the subalgebra   of
$\McNn$ generated by
  $\hat t_1,\ldots,\hat t_n$?
\end{problem}

\bibliographystyle{plain}

\begin{thebibliography}{2}
 
% \bibitem{agumar}
%  S.  Aguzzoli, V. Marra,  Finitely presented MV-algebras with finite automorphism group,
%  J. Logic Comput. 20 (2010)  811--822.
  
  
  \bibitem{bak}
K. A. Baker,
Free vector lattices,
{Canad. J. Math.},
{20} (1968)  58--66.


\bibitem{bey77uno}
W. M. Beynon,
On rational subdivisions of polyhedra with rational vertices,
{Canad. J. Math.},
{ 29} (1977)  238--242.

\bibitem{bey77due}
W. M. Beynon,
Applications of duality in the theory of
finitely generated lattice-ordered abelian groups,
{Canad. J. Math.},
{29} (1977) 243--254.

 \bibitem{mon}
M. Bilen Can, Z. Li, B. Steinberg, Q. Wang (Eds.),
Algebraic Monoids, Group Embeddings, and Algebraic Combinatorics,  Springer,   2014.   
 
 \bibitem{one}
R. D. Byrd, J.T. Lloyd, R. A.Mena,  
On the retractability of some one-relator groups,
{Pacific J. Math.}, 
72  (1977) 351--359.

 \bibitem{cab-forum}
 L. Cabrer, Simplicial geometry of unital lattice-ordered 
abelian groups,  
{Forum Math.},  
27.3 (2015)  1309--1344.
 
 
 \bibitem{cab}
 L. Cabrer, Rational simplicial geometry and projective
 lattice-ordered abelian groups, 
 arXiv:1405.7118v1 [math.RA] 28 May 2014
 
 
 
%\bibitem{cabmun-au} 
%L. Cabrer, D.Mundici,
%Projective MV-algebras and rational 
%polyhedra,  { Algebra Universalis,}
%special issue in memoriam Paul Conrad,
%(J.Mart\' inez et al., Eds.),   
% 62 (2009) 63--74.

\bibitem{cabmun} 
L. Cabrer, D.Mundici,
Rational
polyhedra and projective
lattice-ordered abelian groups with order unit,
{Commun. Contemp.Math.},
14.3  (2012) 1250017 %(20 pages).\\ 
DOI: 10.1142/S0219199712500174.

 

\bibitem{cabmun-jlc} 
L. M. Cabrer, D.Mundici,
A Stone--Weierstrass theorem for
MV-algebras and unital $\ell$-groups, 
J. Logic Comput., 
25.3  (2015)  683-699.

\bibitem{cabmun-etds} 
L. M. Cabrer, D. Mundici,
Classifying orbits of the affine group over the integers,
to appear in 
Ergodic Theory Dynam. Systems,
DOI: 10.1017/etds.2015.45, Published online 22 July 2015


\bibitem{carrus}
O. Caramello, A. C. Russo,  
The Morita-equivalence between MV-algebras and lattice-ordered abelian groups with strong unit, 
J. Algebra,  422 (2015) 752--787.
 
 \bibitem{cigdotmun}
R. Cignoli, I. M. L. D'Ottaviano, D. Mundici, 
Algebraic Foundations of many-valued
Reasoning, Trends in Logic, vol. 7, Kluwer, 
Dordrecht, 2000.


% \bibitem{daw}
%A. A. Daw,   A formula for the number of retracts of finite Boolean algebras,
%Demonstratio Math.,  22  (1989) 929--937.


 

\bibitem{dingripan}
A. Di Nola, R. Grigolia, G. Panti,
Finitely generated free MV-algebras and their automorphism groups,
Studia Logica,  61 (1998) 65--78.

\bibitem{dubpov}
E.  Dubuc, Y.  Poveda,  
Representation theory of MV-algebras,
 Ann. Pure Appl. Logic,
 161 (2010)  1024--1046.

\bibitem{eng}
R. Engelking,  General Topology, Revised and completed edition,
Sigma Series in Pure Mathematics, Vol. 6, 
Herldermann Verlag, Berlin, 1989.

 \bibitem{eva}
D.M. Evans,
Model Theory of Groups and Automorphism Groups,
Cambridge University Press,  1997.   




\bibitem{ewa}
{G. Ewald},
      {Combinatorial convexity and algebraic geometry},
       Graduate Texts in Mathematics, Vol. 168, 
      Springer-Verlag,
      Berlin, Heidelberg,
      1996.

\bibitem{fuc}
L.   Fuchs,  Note on the construction of free MV-algebras. 
Algebra Universalis, 
62 (2009)  45--49.

\bibitem{glamad}
A. M. W. Glass, J. J. Madden, The word problem versus the
isomorphism problem,   
J. Lond. Math. Soc.,  %(2),
30 (1984) 53--61.

 \bibitem{usp}
W. Kubials,  V. Uspenskij, 
A compact group which is not Valdivia compact, 
Proc. Amer. Math. Soc., 
 133.8  (1977)   2483--2487.

\bibitem{marspa}
V. Marra, L. Spada,
Duality, projectivity, and unification in \L ukasiewicz logic and MV-algebras,
Ann. Pure Appl. Logic, 164 (2013) 192--210.

\bibitem{mun-jfa}
D. Mundici, 	Interpretation  of  AF  $C^{*}$-algebras  in
\luk   \ sentential  calculus, 
 J. Funct. Anal.,
65  (1986)      	
15--63.

\bibitem{mun-dcds}
D. Mundici, 
The Haar theorem for lattice-ordered
abelian groups with order-unit,
{Discrete Contin. Dyn. Syst.},
21 (2008) 537--549.
 
  \bibitem{mun11}
D. Mundici, Advanced \L ukasiewicz calculus and MV-algebras,
Trends in Logic, vol. 35, Springer, Berlin, 2011.


\bibitem{mun-cpc}
D. Mundici,
Invariant measure under the affine group
over $\bf Z,$ 
{Combin. Probab. Comput.}, 
 23  (2014)  248--268.
 
   \bibitem{sch}
B. M. Schein, B. Teclezghi, Endomorphisms of finite full
transformation semigroups, 
Proc. Amer. Math. Soc., 
126.9  (1998)  2579--2587.

\bibitem{sht}
M. A. Shtan'ko, 
Markov's theorem
and algorithmically non-recognizable combinatorial manifolds,
{Izv. Math.}, 
 68 (2004) 207--224.
 


\bibitem{sti}
J.  Stillwell, Classical topology and combinatorial group theory, Second edition, 
Graduate Texts in Mathematics, vol. 72, Springer, NY,  1980.



 \bibitem{war}
W. C. Waterhouse,
Retractions of separable commutative algebras,
Arch. Math. (Basel), 60  (1993)  36--39.



\end{thebibliography}
%%%%%%%%%%%%%

 \end{document}